\documentclass[12pt,reqno]{amsart}
\usepackage{amssymb,amsmath,amsthm,hyperref}
\newcommand{\sg}{\textnormal{sg}}

\newtheorem{theorem}{Theorem}
\newtheorem{lemma}[theorem]{Lemma}
\newtheorem{corollary}[theorem]{Corollary}
\newtheorem{proposition}[theorem]{Proposition}

\theoremstyle{remark}

\theoremstyle{definition}

\newtheorem{definition}[theorem]{Definition}

\numberwithin{theorem}{section} \numberwithin{equation}{section}
\numberwithin{example}{section}

\title{On three third order mock theta functions and~Hecke-type~double sums}

\author{Eric Mortenson}

\begin{document}

\date{18 January 2012}



\begin{abstract}
We obtain four Hecke-type double sums for three of Ramanujan's third order mock theta functions.  We discuss how these four are related to the new mock theta functions of Andrews' work on $q$-orthogonal polynomials and Bringmann, Hikami, and Lovejoy's work on unified Witten-Reshetikhin-Turaev invariants of certain Seifert manifolds.  We then prove identities between these new mock theta functions by first expressing them in terms of the universal mock theta function.
\end{abstract}

\address{Department of Mathematics, The University of Queensland,
Brisbane, Australia 4072}
\email{mort@maths.uq.edu.au}
\maketitle
\setcounter{section}{-1}

\section{Notation}\label{section:notation}

 Let $q$ be a nonzero complex number with $|q|<1$ and define $\mathbb{C}^*:=\mathbb{C}-\{0\}$.  We recall some basics:
 \begin{gather*}
(x)_n=(x;q)_n:=\prod_{i=0}^{n-1}(1-q^ix), \ \ (x)_{\infty}=(x;q)_{\infty}:=\prod_{i\ge 0}(1-q^ix),\\
 j(x;q):=(x)_{\infty}(q/x)_{\infty}(q)_{\infty}=\sum_{n}(-1)^nq^{\binom{n}{2}}x^n,\\
 {\text{and }}\ \ j(x_1,x_2,\dots,x_n;q):=j(x_1;q)j(x_2;q)\cdots j(x_n;q).
\end{gather*}
where in the last line the equivalence of product and sum follows from Jacobi's triple product identity.  We also keep in mind the fact that $j(q^n,q)=0$ for $n\in \mathbb{Z}.$  The following are special cases of the above definition.  Here $a$ and $m$ are integers with $m$ positive.  Define
\begin{gather*}
J_{a,m}:=j(q^a;q^m), \ \ J_m:=J_{m,3m}=\prod_{i\ge 1}(1-q^{mi}), \ {\text{and }}\overline{J}_{a,m}:=j(-q^a;q^m).
\end{gather*}
\section{Introduction}\label{section:intro}
Historically, mock theta functions have many forms of representation:  Eulerian forms, Hecke-type double sums, Appell-Lerch sums and Fourier coefficients of meromorphic Jacobi forms.  Recently they have been cast as holomorphic parts of weak Maass forms.  With the exception of a Hecke-type double sum for the third order mock theta function $\psi(q)$ found in Andrews' recent work on $q$-orthogonal polynomials \cite{A2}, Hecke-type double sum representations for third order mock theta functions are unknown.  Here we obtain Hecke-type sums for the third order functions $1+2\psi(q)$, $\nu(-q)$, $\phi(q)$, and $\nu(q)$.   Where these representations fit with respect to Zwegers' modularity theory \cite{Zw} is also addressed.  In the process, this leads us to two new mock theta functions found in Andrews' work \cite{A2} on $q$-orthogonal polynomials and to the two new mock theta functions found in Bringmann, Hikami, and Lovejoy's work \cite{BHL} on unified Witten-Reshetikhin-Turaev invariants of certain Seifert manifolds.  We obtain expressions for the new mock theta functions of \cite{A2,BHL} in terms of the universal mock theta function
\begin{equation}
g(x,q):=x^{-1}\Big ( -1 +\sum_{n=0}^{\infty}\frac{q^{n^2}}{(x)_{n+1}(q/x)_{n}} \Big ),
\end{equation}
and use this information to prove identities between the new mock theta functions.

All of the results in this paper can be shown using Theorems $1.4$ and $1.6$ of \cite{HM}, but for variety and brevity we will use other techniques on occasion.

We first recall some notation which will allow us to state Theorems $1.4$ and $1.6$ of \cite{HM}.  We emphasize that it is Theorem $1.6$ of \cite{HM} which guides us to representing the mock theta functions of \cite{A2,BHL} in terms of $g(x,q)$.  We will use the following definition of an Appell-Lerch sum:
\begin{definition}  \label{definition:mdef} Let $x,z\in\mathbb{C}^*$ with neither $z$ nor $xz$ an integral power of $q$. Then
\begin{equation}
m(x,q,z):=\frac{1}{j(z;q)}\sum_{r=-\infty}^{\infty}\frac{(-1)^rq^{\binom{r}{2}}z^r}{1-q^{r-1}xz}.\label{equation:mdef-eq}
\end{equation}
\end{definition}
\noindent In \cite[Proposition $4.2$]{HM}, we showed that it is an easy consequence of \cite[Theorem $2.2$]{H1} that
\begin{equation}
g(x,q)=-x^{-1}m(q^2x^{-3},q^3,x^2)-x^{-2}m(qx^{-3},q^3,x^2).\label{equation:g-to-m}
\end{equation}

\noindent We recall the following notation for a special type of Hecke-type double sum:  
\begin{definition} \label{definition:fabc-def}  Let $x,y\in\mathbb{C}^*$ and define $\sg (r):=1$ for $r\ge 0$ and $\sg(r):=-1$ for $r<0$. Then
\begin{equation*}
f_{a,b,c}(x,y,q):=\sum_{\substack{\sg (r)=\sg(s)}} \sg(r)(-1)^{r+s}x^ry^sq^{a\binom{r}{2}+brs+c\binom{s}{2}}.\label{definition:f-def}\\
\end{equation*}
\end{definition}

\noindent We also define the following expression involving Appell-Lerch sums:
\begin{align}
g_{a,b,c}&(x,y,q,z_1,z_0)\label{equation:mdef-2}\\
&:=\sum_{t=0}^{a-1}(-y)^tq^{c\binom{t}{2}}j(q^{bt}x;q^a)m\Big (-q^{a\binom{b+1}{2}-c\binom{a+1}{2}-t(b^2-ac)}\frac{(-y)^a}{(-x)^b},q^{a(b^2-ac)},z_0\Big )\notag\\
&\ \ \ \ \ +\sum_{t=0}^{c-1}(-x)^tq^{a\binom{t}{2}}j(q^{bt}y;q^c)m\Big (-q^{c\binom{b+1}{2}-a\binom{c+1}{2}-t(b^2-ac)}\frac{(-x)^c}{(-y)^b},q^{c(b^2-ac)},z_1\Big ).\notag
\end{align}
We have
\begin{theorem}[\cite{HM}, Theorem $1.6$]   \label{theorem:masterFnp} Let $n$ and $p$ be positive integers with $(n,p)=1$.  For generic $x,y\in \mathbb{C}^*$
\begin{align*}
f_{n,n+p,n}(x,y,q)=g_{n,n+p,n}(x,y,q,-1,-1)+\theta_{n,p}(x,y,q),
\end{align*}
where
\begin{align*}
&\theta_{n,p}(x,y,q):=\sum_{r^*=0}^{p-1}\sum_{s^*=0}^{p-1}q^{n\binom{r-(n-1)/2}{2}+(n+p)\big (r-(n-1)/2\big )\big (s+(n+1)/2\big )+n\binom{s+(n+1)/2}{2}}(-x)^{r-(n-1)/2}\\
& 
 \cdot \frac{(-y)^{s+(n+1)/2}J_{p^2(2n+p)}^3 j(-q^{np(s-r)}x^n/y^n;q^{np^2})j(q^{p(2n+p)(r+s)+p(n+p)}x^py^p;q^{p^2(2n+p)})}{\overline{J}_{0,np(2n+p)}j(q^{p(2n+p)r+p(n+p)/2}(-y)^{n+p}/(-x)^n,q^{p(2n+p)s+p(n+p)/2}(-x)^{n+p}/(-y)^n;q^{p^2(2n+p)})}.\notag
\end{align*}
Here $r:=r^*+\{(n-1)/2\}$ and $s:=s^*+\{ (n-1)/2\}$, with $0\le \{ \alpha\}<1$ denoting the fractional part of $\alpha$. 
\end{theorem}

\noindent We also have
\begin{theorem} [\cite{HM}, Theorem $1.4$] \label{theorem:main-acdivb}Let $a,b,$ and $c$ be positive integers with $ac<b^2$ and $b$ divisible by $a$ and $c$. Then
\begin{align*}
& f_{a,b,c}(x,y,q)=h_{a,b,c}(x,y,q,-1,-1)-\frac{1}{\overline{J}_{0,b^2/a-c}\overline{J}_{0,b^2/c-a}}\cdot \theta_{a,b,c}(x,y,q),
\end{align*}
where 
\begin{align*}
h_{a,b,c}(x,y,q,z_1,z_0):&=j(x;q^a)m\Big( -q^{a\binom{b/a+1}{2}-c}{(-y)}{(-x)^{-b/a}},q^{b^2/a-c},z_1 \Big )\\
& \ \ \ \ \ \ +j(y;q^c)m\Big( -q^{c\binom{b/c+1}{2}-a}{(-x)}{(-y)^{-b/c}},q^{b^2/c-a},z_0 \Big ),
\end{align*}
and
\begin{align*}
&\theta_{a,b,c}(x,y,q):=\sum_{d=0}^{b/c-1}\sum_{e=0}^{b/a-1}\sum_{f=0}^{b/a-1}
q^{(b^2/a-c)\binom{d+1}{2}+(b^2/c-a)\binom{e+f+1}{2}+a\binom{f}{2}}j\big (q^{(b^2/a-c)(d+1)+bf}y;q^{b^2/a}\big ) \\
&\ \ \ \ \ \ \ \ \ \  \cdot (-x)^{f}j\big (q^{b(b^2/(ac)-1)(e+f+1)-(b^2/a-c)(d+1)+b^3(b-a)/(2a^2c)}(-x)^{b/a}y^{-1};q^{(b^2/a)(b^2/(ac)-1)}\big ) \\
& \cdot \frac{J_{b(b^2/(ac)-1)}^3j\big (q^{ (b^2/c-a)(e+1)+(b^2/a-c)(d+1)-c\binom{b/c}{2}-a\binom{b/a}{2}}(-x)^{1-b/a}(-y)^{1-b/c};q^{b(b^2/(ac)-1)}\big )}
{j\big (q^{(b^2/c-a)(e+1)-c\binom{b/c}{2}}(-x)(-y)^{-b/c},q^{(b^2/a-c)(d+1)-a\binom{b/a}{2}}(-x)^{-b/a}(-y);q^{b(b^2/(ac)-1)}\big )}.
\end{align*}
\end{theorem}

Andrews \cite[$(1.10)$]{A2} showed the following for Ramanujan's third order mock theta function $\psi(q)$:
\begin{equation*}
1+\psi(q):=1+\sum_{n=1}^{\infty}\frac{q^{n^2}}{(q;q^2)_{n}}=\frac{1}{(q)_{\infty}}\sum_{n=0}^{\infty}(-1)^nq^{2n^2+n}(1-q^{6n+6})\sum_{j=0}^nq^{-\binom{j+1}{2}}.
\end{equation*}
This motivates
 \begin{theorem} \label{theorem:newheckethm}The third order mock theta functions $\psi(q)$ and $\nu(q)$ have the following Hecke-type double sum representations:
 \begin{align}
 1+2\psi(q)&=\frac{1}{(q)_{\infty}}\cdot \sum_{n= 0}^{\infty}(-1)^nq^{2n^2+n}(1+q^{2n+1})\sum_{j=-n}^nq^{-\binom{j+1}{2}},\label{equation:psi3(q)-hecke}\\
 \nu(-q)&=\frac{1}{(q)_{\infty}}\cdot \sum_{n=0}^{\infty}(-1)^nq^{2n^2+2n}\sum_{j=-n}^nq^{-\binom{j+1}{2}}.\label{equation:nu3(-q)-hecke}
 \end{align}
 \end{theorem}
 \noindent In \cite[$(1.14)$, $(1.15)$]{A2}, Andrews also showed the following for two new mock theta functions:
 \begin{align}
\overline{\psi}_0(q):= \sum_{n=0}^{\infty}\frac{q^{2n^2}}{(-q;q)_{2n}}&=\frac{1}{(q^2;q^2)_{\infty}}\sum_{n=0}^{\infty}q^{4n^2+n}(1-q^{6n+3})\sum_{j=-n}^n(-1)^jq^{-j^2},\label{equation:A278-1.14}\\
 \overline{\psi}_1(q):= \sum_{n=0}^{\infty}\frac{q^{2n^2+2n}}{(-q;q)_{2n+1}}&=\frac{1}{(q^2;q^2)_{\infty}}\sum_{n=0}^{\infty}q^{4n^2+3n}(1-q^{2n+1})\sum_{j=-n}^n(-1)^jq^{-j^2}.\label{equation:A278-1.15}
 \end{align}
These four functions are related.  Indeed, the four functions $1+2\psi(q),$ $\nu(-q)$, $\overline{\psi}_0(q)$, $\overline{\psi}_1(q)$ form a vector-valued mock theta function not unlike that for the fifth order mock theta functions $f_0(q)$, $f_1(q)$, $F_0(q)$, $F_1(q)$ as found in \cite{Zw}.  It turns out that Andrews' two new mock theta functions can be written in terms of the third order mock theta function $\phi(q)$, where
\begin{equation*}
\phi(q):=\sum_{n= 0}^{\infty}\frac{q^{n^2}}{(-q^2;q^2)_n}.
\end{equation*}
A straightforward exercise with Eulerian forms reveals that
\begin{align}
2\overline{\psi}_0(q^2)&= \phi(q)+\phi(-q),\label{equation:A278-1.14phi}\\
2q \overline{\psi}_1(q^2)&=\phi(q)-\phi(-q).\label{equation:A278-1.15phi}
 \end{align}

We note that the third order mock theta functions can all be written in terms of $g(x,q)$ \cite{W3}.  In \cite{H1}, Hickerson proved the mock theta conjectures.  These are identities which express the fifth order mock theta functions in terms of the universal mock theta function $g(x,q)$ and theta functions.  In \cite{H2}, Hickerson found and then proved analogous identities for the seventh order functions.  Here we prove similar identities for Andrews' two new mock theta functions of (\ref{equation:A278-1.14}) and (\ref{equation:A278-1.15}):
 \begin{theorem} \label{theorem:newmockthetaid} The following identities are true:
 \begin{align}
\overline{\psi}_0(q)&=2-2qg(-q,q^8)-\frac{J_{1,2}{\overline{J}}_{3,8}}{J_2},\label{equation:1.14mockthetaid}\\
\overline{\psi}_1(q)&=2q^2g(-q^3,q^8)+\frac{J_{1,2}{\overline{J}}_{1,8}}{J_2}.\label{equation:1.15mockthetaid}
 \end{align}
 \end{theorem}
We sketch how one is led to such identities.  Once one has the Hecke form of the mock theta function, one uses Theorem \ref{theorem:masterFnp} and basic Appell-Lerch sum properties as a guide to produce an expression like identity (\ref{equation:g-to-m}).  In the process, it is best to ignore the theta functions as well as the $z$ part of the $m(x,q,z)$ terms.  What is left is a theta function, so one uses a software package such as Maple or Mathematica to determine if the theta function has a nice form.  For both $\overline{\psi}_0(q)$ and $\overline{\psi}_1(q),$ that is the case.

Vector-valued mock theta functions tend to come in pairs.  The above four functions can all be written in terms of $f_{3,5,3}(x,y,q)$'s, so \cite{Zw} suggests that the paired vector might consist of functions which can be written in terms of $f_{1,7,1}(x,y,q)$'s.  How one goes about finding such a pair is not obvious.  Sometimes, mock theta functions are sign flips away from a theta function.  So with this in mind, we recall the following identity which is found in Andrews \cite[$(1.2)$]{A0} as well as Kac and Peterson \cite[$(5.19)$]{KP}: 
\begin{align}
J_1^2&=\sum_{n= 0}^{\infty}q^{2n^2+n}(1-q^{2n+1})\sum_{j=-n}^n(-1)^jq^{-3j^2/2+j/2}\\
&=\Big ( \sum_{\substack{n+j\ge 0\\n-j\ge 0}}- \sum_{\substack{n+j< 0\\n-j< 0}} \Big )(-1)^jq^{2n^2+n-3j^2/2+j/2}=f_{1,7,1}(q,q^2,q)-qf_{1,7,1}(q^3,q^4,q).\notag
\end{align}
Making some judicious sign flips, we find that 
\begin{align}
\overline{J}_{1,4}\cdot \phi(q)&=\sum_{n= 0}^{\infty}(-1)^{n}q^{2n^2+n}(1+q^{2n+1})\sum_{j=-n}^n(-1)^jq^{-3j^2/2+j/2}\label{equation:hecke-phi}\\
&=\Big ( \sum_{\substack{n+j\ge 0\\n-j\ge 0}}- \sum_{\substack{n+j< 0\\n-j< 0}} \Big )(-1)^{n+j}q^{2n^2+n-3j^2/2+j/2}\notag\\
&=f_{1,7,1}(-q,-q^2,q)+qf_{1,7,1}(-q^3,-q^4,q).\notag
\end{align}
To find the other components of the vector-valued mock theta function, Zwegers' thesis \cite{Zw} leads us to
{\allowdisplaybreaks \begin{align}
\overline{J}_{1,4}\cdot \nu(q)&=\sum_{n= 0}^{\infty}(-1)^nq^{2n^2+2n}\sum_{j=-n}^n(-1)^jq^{-3j^2/2+j/2}\label{equation:hecke-nu}\\
&=\frac{1}{2}\Big ( \sum_{\substack{n+j\ge 0\\n-j\ge 0}}- \sum_{\substack{n+j< 0\\n-j< 0}} \Big )(-1)^{n+j}q^{2n^2+2n-3j^2/2+j/2},\notag\\
J_{1,2}\cdot \Big ( q\overline{\phi}_0(q)+1\Big )&=\sum_{n= 0}^{\infty}q^{4n^2+n}(1-q^{6n+3})\sum_{j=-n}^n(-1)^jq^{-3j^2-j}\label{equation:hecke-phibar0}\\
&=\Big ( \sum_{\substack{n+j\ge 0\\n-j\ge 0}}- \sum_{\substack{n+j< 0\\n-j< 0}} \Big )(-1)^{j}q^{4n^2+n-3j^2-j},\notag\\
J_{1,2}\cdot \overline{\phi}_1(q)&=\sum_{n= 0}^{\infty}q^{4n^2+3n}(1-q^{2n+1})\sum_{j=-n}^n(-1)^jq^{-3j^2-j}\label{equation:hecke-phibar1}\\
&=\Big ( \sum_{\substack{n+j\ge 0\\n-j\ge 0}}- \sum_{\substack{n+j< 0\\n-j< 0}} \Big )(-1)^{j}q^{4n^2+3n-3j^2-j}.\notag
\end{align}}%
The last two are two new mock theta functions of Bringmann, Hikami, and Lovejoy \cite{BHL}, where
\begin{align*}
\overline{\phi}_{0}(q):&=\sum_{n= 0}^{\infty}q^n(-q)_{2n+1} {\text{ and }}\overline{\phi}_{1}(q):=\sum_{n= 0}^{\infty}q^n(-q)_{2n}.
\end{align*}
\noindent We also note that (\ref{equation:hecke-phibar0}) is a slightly rewritten \cite[$(2.7)$] {BHL}.  We will prove

\begin{theorem}\label{theorem:four-new-heckes} Identities (\ref{equation:hecke-phi}) - (\ref{equation:hecke-phibar1}) are true.
\end{theorem}

We will also prove identities between Andrews' two new mock theta functions and Bringmann, Hikami, and Lovejoy's new mock theta functions.  We recall that Bringmann, Hikami, and Lovejoy also proved
\begin{align}
2q^2\overline{\phi}_{0}(q^2)&=\psi(q)+\psi(-q),\label{equation:BHLpsieven}\\
2q\overline{\phi}_{1}(q^2)&=\psi(q)-\psi(-q).\label{equation:BHLpsiodd}
\end{align}
We express $\overline{\phi}_0(q)$ and $\overline{\phi}_1(q)$ in terms of $g(x,q).$
\begin{theorem} \label{theorem:BHLids} The following identities are true:
\begin{align}
q\overline{\phi}_{0}(q)&=-1+qg(-q,q^{8})+\frac{\overline{J}_{2,4}\overline{J}_{3,8}}{J_2},\label{equation:BHLphibar0}\\
\overline{\phi}_{1}(q)&=-q^2g(-q^3,q^{8})+\frac{\overline{J}_{2,4}\overline{J}_{1,8}}{J_2}.\label{equation:BHLphibar1}
\end{align}
\end{theorem}
\noindent The above two identities were found using Theorem \ref{theorem:masterFnp} as a guide.  Here the respective Hecke-type forms are in terms of $f_{1,7,1}(x,y,q)$'s.  We could use Theorem \ref{theorem:masterFnp} to prove these two identities, but for brevity, we will use new results which follow from Appell-Lerch sum properties of \cite{HM}.

Using Theorems \ref{theorem:newmockthetaid} and \ref{theorem:BHLids}, we then have the following immediate corollary which relates the mock theta functions of Andrews \cite{A2} to those of Bringmann, Hikami, and Lovejoy \cite{BHL}.
\begin{corollary}  The following identities are true:
\begin{align}
\overline{\psi}_0(q)+2q\overline{\phi}_0(q)&=-\frac{\overline{J}_{3,8}}{J_2}\cdot\Big (J_{1,2}-2\overline{J}_{2,4}\Big ),\\
\overline{\psi}_1(q)+2\overline{\phi}_1(q)&=\frac{\overline{J}_{1,8}}{J_2}\cdot\Big (J_{1,2}+2\overline{J}_{2,4}\Big ).
\end{align}
\end{corollary}

In Section \ref{section:prelim}, we recall useful facts covering theta function identities, Appell-Lerch sum properties, Hecke-type double sums, and third order mock theta functions.  We also recall and give new proofs of properties found in \cite{RLN} for the universal mock theta function $g(x,q)$.  In Section \ref{section:proof-newheckethm}, we prove Theorem \ref{theorem:newheckethm}.  Here we prove identity (\ref{equation:psi3(q)-hecke}) with Theorem \ref{theorem:masterFnp}.  We could use Theorem \ref{theorem:masterFnp} to prove identity (\ref{equation:nu3(-q)-hecke}), but for variety, we use a different technique.  One could also use Corollary $6$ of \cite{A2}.  Theorem \ref{theorem:newmockthetaid} is shown in Section \ref{section:proof-newmockthetaid}.  Here we use Theorem \ref{theorem:masterFnp} for both identities.  In Section \ref{section:proofs-four-new-heckes} we prove Theorem \ref{theorem:four-new-heckes}.  For identities (\ref{equation:hecke-phi}) and (\ref{equation:hecke-nu}) we rewrite the Hecke-type doube sums and use Theorem \ref{theorem:main-acdivb}.  For identities (\ref{equation:hecke-phibar0}) and (\ref{equation:hecke-phibar1}) we use the Bailey pair techniques of \cite{A1}.  In Section \ref{section:proof-BHLids}, we prove Theorem \ref{theorem:BHLids}.


\section{Preliminaries}\label{section:prelim}
\subsection{Properties of theta functions}\label{section:thetaprelim}

For later use, we state the following easily shown identities:
\begin{gather*}
\overline{J}_{0,1}=2\overline{J}_{1,4}=\frac{2J_2^2}{J_1},  \overline{J}_{1,2}=\frac{J_2^5}{J_1^2J_4^2},   J_{1,2}=\frac{J_1^2}{J_2},   \overline{J}_{1,3}=\frac{J_2J_3^2}{J_1J_6}, \\
 J_{1,4}=\frac{J_1J_4}{J_2},   J_{1,6}=\frac{J_1J_6^2}{J_2J_3},   \overline{J}_{1,6}=\frac{J_2^2J_3J_{12}}{J_1J_4J_6}.
\end{gather*}
We state additional identities:
\begin{subequations}
\begin{equation}
j(q^n x;q)=(-1)^nq^{-\binom{n}{2}}x^{-n}j(x;q), \ \ n\in\mathbb{Z},\label{equation:1.8}
\end{equation}
\begin{equation}
j(x;q)=j(q/x;q)=-xj(x^{-1};q)\label{equation:1.7},
\end{equation}
\begin{equation}
j(-x;q)={J_{1,2}j(x^2;q^2)}/{j(x;q)} \ \ {\text{if $x$ is not an integral power of $q$,}}\label{equation:1.9}
\end{equation}
\begin{equation}
j(x;q)={J_1}j(x,qx,\dots,q^{n-1}x;q^n)/{J_n^n} \ \ {\text{if $n\ge 1$,}}\label{equation:1.10}
\end{equation}
\begin{equation}
j(x;-q)={j(x;q^2)j(-qx;q^2)}/{J_{1,4}},\label{equation:1.11}
\end{equation}
\begin{equation}
j(z;q)=\sum_{k=0}^{m-1}(-1)^k q^{\binom{k}{2}}z^k
j\big ((-1)^{m+1}q^{\binom{m}{2}+mk}z^m;q^{m^2}\big ),\label{equation:jsplit}
\end{equation}
\begin{equation}
j(x^n;q^n)={J_n}j(x,\zeta_nx,\cdots, \zeta_n^{n-1}x;q^n)/{J_1^n} \ \ {\text{if $n\ge 1$}}\label{equation:1.12}
\end{equation}
\end{subequations}
\noindent where and $\zeta_n$ an $n$-th primitive root of unity.  We recall the classical partial fraction  expansion for the reciprocal of Jacobi's theta product
\begin{equation}
\sum_n\frac{(-1)^nq^{\binom{n+1}{2}}}{1-q^nz}=\frac{J_1^3}{j(z;q)},\label{equation:Reciprocal}
\end{equation}

\noindent where $z$ is not an integral power of $q$.  A convenient form of the Riemann relation for theta functions is
\begin{proposition}\label{proposition:CHcorollary} For generic $a,b,c,d\in \mathbb{C}^*$
\begin{align*}
j(ac,a/c,bd,b/d;q)=j(ad,a/d,bc,b/c;q)+b/c \cdot j(ab,a/b,cd,c/d;q).
\end{align*}
\end{proposition}

\noindent We collect several useful results about theta functions in terms of a proposition:  
\begin{proposition}   For generic $x,y,z\in \mathbb{C}^*$ 
\begin{subequations}
\begin{equation}
j(qx^3;q^3)+xj(q^2x^3;q^3)=j(-x;q)j(qx^2;q^2)/J_2={J_1j(x^2;q)}/{j(x;q)},\label{equation:H1Thm1.0}
\end{equation}
\begin{equation}
j(x;q)j(y;q)=j(-xy;q^2)j(-qx^{-1}y;q^2)-xj(-qxy;q^2)j(-x^{-1}y;q^2),\label{equation:H1Thm1.1}
\end{equation}
\begin{equation}
j(-x;q)j(y;q)-j(x;q)j(-y;q)=2xj(x^{-1}y;q^2)j(qxy;q^2),\label{equation:H1Thm1.2A}
\end{equation}
\begin{equation}
j(-x;q)j(y;q)+j(x;q)j(-y;q)=2j(xy;q^2)j(qx^{-1}y;q^2),\label{equation:H1Thm1.2B}
\end{equation}
\end{subequations}
\end{proposition}

\noindent Identity (\ref{equation:H1Thm1.0}) is the quintuple product identity. 

Finally, we recall a fact which follows immediately from \cite[Lemma $2$]{ASD} and is also \cite[Theorem $1.7$]{H1}.
\begin{proposition}\label{proposition:H1Thm1.7} Let $C$ be a nonzero complex number, and let $n$ be a nonnegative integer.  Suppose that $F(z)$ is analytic for $z\ne 0$ and satisfies $F(qz)=Cz^{-n}F(z)$.  Then either $F(z)$ has exactly $n$ zeros in the annulus $|q|<|z|\le 1$ or $F(z)=0$ for all $z$.
\end{proposition}
We will need the following four identities, which appear to be new.

\begin{proposition}\label{proposition:hecke-phi-id}  Let $x\ne0.$  Then
\begin{equation}j(q^2x;q^4) j(q^5x;q^8)+\frac{q}{x}\cdot j(x;q^4) j(qx;q^8)-\frac{J_1}{J_4}\cdot j(-q^3x;q^4)j(q^3x;q^8)=0
\label{equation:hecke-phi-id}
\end{equation}
\end{proposition}
\begin{proof}[Proof of Proposition \ref{proposition:hecke-phi-id}]  Let $f(x)$ be the left hand side of (\ref{equation:hecke-phi-id}).  It satisfies $f(q^8x)=-q^{-13}x^{-3}f(x).$  By Proposition \ref{proposition:H1Thm1.7}, if $f$ has more than $3$ zeros in $|q^8|<|x|\le 1$, then $f(x)=0$ for all nonzero $x$.  But it is easy to check that $f(x)=0$ for $x=1, q^2, q^3, q^{7}$.
\end{proof}

\begin{proposition}\label{proposition:newgenid} Let $x\ne0.$  Then
\begin{equation}
J_{12}^3j(q^2x;q^3)j(-qx^2;q^6)+xJ_{12}^3j(qx;q^3)j(-q^5x^2;q^6)-J_{2,12}^2J_4j(-x;q^3)j(q^3x^2;q^6)=0\label{equation:newgenid}
\end{equation}
\end{proposition}
\begin{proof}[Proof of Proposition \ref{proposition:newgenid}]  Let $f(x)$ be the left hand side of (\ref{equation:newgenid}).  It satisfies $f(q^3x)=-q^{-3}x^{-3}f(x).$  By Proposition \ref{proposition:H1Thm1.7}, if $f$ has more than $3$ zeros in $|q^3|<|x|\le 1$, then $f(x)=0$ for all nonzero $x$.  But it is easy to check that $f(x)=0$ for $x=-1,$ $\pm i q^{1/2}$, $q$, $\pm q^{3/2}$, $q^2$, $\pm q^{5/2}$.
\end{proof}

\begin{proposition}\label{lemma:newids} We have 
\begin{align}
J_{3,6}J_{2,16}^2 -J_{4,8}J_{3,24}J_{1,8}&=qJ_{1,2}J_{2,8}\overline{J}_{24,96}\label{equation:And1.14},\\
-J_{3,6}J_{6,16}^2 +J_{4,8}J_{9,24}J_{3,8}&=q^3J_{1,2}J_{2,8}\overline{J}_{24,96}\label{equation:And1.15}.
\end{align}
\end{proposition}

\noindent To prove these identities, we first give a lemma.
\begin{lemma}{\label{lemma:0}}
Let $z\ne0.$  Then
\begin{gather}
J_{12,24}j(q^4z^4;q^8)-j(q^3z^3;q^6)j(-q^4z^2;q^8)-q^3z^{-3}j(qz;q^2)j(-z^6;q^{24})=0\label{equation:partA},\\
J_{2,8}j(q^{12}z^2;q^{24})-\overline{J}_{2,8}j(q^3z;q^6)+q^2j(qz;q^8)j(q^3z^{-1};q^{24})+q^2j(qz^{-1};q^8)j(q^3z;q^{24})=0.\label{equation:partB}
\end{gather}
\end{lemma}
\begin{proof}[Proof of Lemma \ref{lemma:0}]We prove (\ref{equation:partA}).  Let $f(z)$ be the left hand side of (\ref{equation:partA}).  It satisfies $f(q^4z)=q^{-16}z^{-8}f(z).$ By Proposition \ref{proposition:H1Thm1.7}, if $f$ has more than $8$ zeros in $|q^4|<|z|\le 1$, then $f(z)=0$ for all nonzero $z$.  The first term in (\ref{equation:partA}) is zero for $z=\pm q$, $\pm iq$, $\pm q^3$, $\pm iq^3$, and the second term is zero for $z=iq^2$ (among others) so checking that $f(z)=0$ for these $9$ values just involves product rearrangements of the other two terms.  Here we use facts such as $j(i,q)=(1-i)J_{1,4}$ and $j(iq,q^2)=J_{4,8}.$

We prove (\ref{equation:partB}).  Let $f(z)$ be the left side of (\ref{equation:partB}); it satisfies $f(q^{24})z=q^{-48}z^{-4}f(z).$  So if it has more than $4$ zeros in $|q^{24}|<|z|\le 1,$ then $f(z)=0$ for all $z$.  It is easy to see that $f(z)=0$ for $z=q^3$, $q^9$, $q^{15}$, and $q^{21}$; in each case the second term and one of the last two terms in $f(z)$ is zero.  Also,
\begin{align*}
f(q)&=J_{2,8}J_{10,24}-\overline{J}_{2,8}J_2+q^2J_{2,8}J_{2,24}+0=J_{2,8}(J_{10,24}+q^2J_{2,24})-\overline{J}_{2,8}J_2\\
&=J_{2,8}j(-q^2;-q^6)-\overline{J}_{2,8}J_2=0.&(\text{by }(\ref{equation:jsplit}))
\end{align*}
So $f(z)=0$ for all $z.$
\end{proof}

\begin{proof}[Proof of Proposition \ref{lemma:newids}]  We will use the following two identities
\begin{align}
2q^{2}J_{2,16}^2&=\overline{J}_{2,8}J_{4,8}-J_{2,8}\overline{J}_{4,8}\label{equation:preidA},\\
2J_{6,16}^2&=\overline{J}_{2,8}J_{4,8}+J_{2,8}\overline{J}_{4,8}\label{equation:preidB},
\end{align}
which follow respectively from (\ref{equation:H1Thm1.2A}) and (\ref{equation:H1Thm1.2B})with $q\rightarrow q^8$, $x=q^2$, $y=q^{4}$, .  

We prove (\ref{equation:And1.14}).  Letting $z=1$ in (\ref{equation:partA}) gives
\begin{align}
0&=J_{12,24}J_{4,8}-J_{3,6}\overline{J}_{4,8}-q^3J_{1,2}\overline{J}_{0,24}=J_{12,24}J_{4,8}-J_{3,6}\overline{J}_{4,8}-2q^3J_{1,2}\overline{J}_{24,96}.\label{equation:And1.14A}
\end{align}
Multiply by $J_{2,8}$ and rearrange to obtain
\begin{align}
2q^3J_{2,8}J_{1,2}\overline{J}_{24,96}=J_{2,8}J_{12,24}J_{4,8}-J_{3,6}J_{2,8}\overline{J}_{4,8}.\label{equation:And1.14B}
\end{align}
Letting $z=1$ in (\ref{equation:partB}) gives
\begin{align}
J_{2,8}J_{12,24}-\overline{J}_{2,8}J_{3,6}+2q^2J_{1,8}J_{3,24}=0.\label{equation:And1.14C}
\end{align}
Multiply by $J_{4,8}$ and rearrange to obtain
\begin{align}
J_{2,8}J_{12,24}{J}_{4,8}=J_{3,6}\overline{J}_{2,8}J_{4,8}-2q^2{J}_{1,8}J_{3,24}J_{4,8}.\label{equation:And1.14D}
\end{align}
Substitute (\ref{equation:And1.14D}) into (\ref{equation:And1.14B}) and use (\ref{equation:preidA}):
\begin{align*}
2q^3J_{2,8}J_{1,2}\overline{J}_{24,96}&=J_{3,6}\overline{J}_{2,8}J_{4,8}-2q^2{J}_{1,8}J_{3,24}J_{4,8}-J_{3,6}J_{2,8}\overline{J}_{4,8}\\
&=J_{3,6}2q^2J_{2,16}^2-2q^2{J}_{1,8}J_{3,24}J_{4,8}.
\end{align*}
Dividing by $2q^2$ yields (\ref{equation:And1.14}).

We prove (\ref{equation:And1.15}).  Setting $z=q^{12}$ in (\ref{equation:partB}), we have
\begin{align}
J_{2,8}J_{36,24}-\overline{J}_{2,8}J_{15,6}+q^2J_{13,8}J_{-9,24}+q^2J_{-11,8}J_{15,25}=0,\label{equation:And1.15A}
\end{align}
which can be rewritten
\begin{align}
2J_{3,8}J_{9,24}=J_{2,8}J_{12,24}+\overline{J}_{2,8}J_{3,6}.\label{equation:And1.15B}
\end{align}
Multiply by $J_{4,8}$ and rearrange to obtain
\begin{align}
J_{2,8}J_{12,24}J_{4,8}=-J_{3,6}\overline{J}_{2,8}J_{4,8}+2J_{3,8}J_{9,24}J_{4,8}.\label{equation:And1.15C}
\end{align}
Substitute this into (\ref{equation:And1.14B}) and use (\ref{equation:preidB}):
\begin{align*}
2q^3J_{2,8}J_{1,2}\overline{J}_{24,96}&=-J_{3,6}\overline{J}_{2,8}J_{4,8}+2J_{3,8}J_{9,24}J_{4,8}-J_{3,6}J_{2,8}\overline{J}_{4,8}\\
&=-J_{3,6}2J_{6,16}^2+2J_{3,8}J_{9,24}J_{4,8}.
\end{align*}
Dividing by $2$ gives (\ref{equation:And1.15}).
\end{proof}


\subsection{Properties of the Appell-Lerch sums}\label{section:prop-mxqz}

The Appell-Lerch sum $m(x,q,z)$ satisfies several functional equations and identities, which we collect in the form of a proposition. 

\begin{proposition}  For generic $x,z,z_0,z_1\in \mathbb{C}^*$
\begin{subequations}
\begin{equation}
m(x,q,z)=m(x,q,qz),\label{equation:m-fnq-z}
\end{equation}
\begin{equation}
m(x,q,z)=x^{-1}m(x^{-1},q,z^{-1}),\label{equation:m-fnq-flip}
\end{equation}
\begin{equation}
m(qx,q,z)=1-xm(x,q,z),\label{equation:m-fnq-x}
\end{equation}
\begin{equation}
m(x,q,z)=1-q^{-1}xm(q^{-1}x,q,z)\label{equation:m-fnq-x-alt1},
\end{equation}
\begin{equation}
m(x,q,z)=x^{-1}-x^{-1}m(qx,q,z), \label{equation:m-fnq-x-alt2}
\end{equation}
\begin{equation}
m(x,q,z_1)-m(x,q,z_0)=\frac{z_0J_1^3j(z_1/z_0;q)j(xz_0z_1;q)}{j(z_0;q)j(z_1;q)j(xz_0;q)j(xz_1;q)},\label{equation:m-change-z}
\end{equation}
\end{subequations}
\end{proposition}
\noindent   The proofs are straightforward and will be omitted.  Although one can find most of these in \cite{L1,L2}, these papers are hard to obtain.   In addition, the German summary \cite{L2} has a few typos.  The equivalent of (\ref{equation:m-change-z}), for example, reads
\begin{equation*}
m(x,q,z_1)=m(x,q,z_0)=\frac{z_0J_1^3j(z_1/z_0;q)j(xz_0z_1;q)}{j(z_0;q)j(z_1;q)j(xz_0;q)j(xz_1;q)}.
\end{equation*}
A modern list of Appell-Lerch sum properties with proofs can be found in \cite{Zw}.

We recall a useful result:
\begin{theorem} [\cite{HM}, Theorem $3.6$]\label{theorem:msplit} For generic $x,z,z'\in \mathbb{C}^*$ 
\begin{align*}
m(&x,q,z) = \sum_{r=0}^{n-1} q^{{-\binom{r+1}{2}}} (-x)^r m\big(-q^{{\binom{n}{2}-nr}} (-x)^n, q^{n^2}, z' \big)\notag\\
& + \frac{z' J_n^3}{j(xz;q) j(z';q^{n^2})}  \sum_{r=0}^{n-1}
\frac{q^{{\binom{r}{2}}} (-xz)^r
j\big(-q^{{\binom{n}{2}+r}} (-x)^n z z';q^n\big)
j(q^{nr} z^n/z';q^{n^2})}
{ j\big(-q^{{\binom{n}{2}}} (-x)^n z', q^r z;q^n\big )}.
\end{align*}
\end{theorem}
\noindent Identity (\ref{equation:1.8}) easily yields two $n=2$ specializations:

\begin{corollary} \label{corollary:msplitn2zprime} For generic $x,z,z'\in \mathbb{C}^*$ 
{\allowdisplaybreaks \begin{align}
m&(x,q,z)=m (-qx^2,q^4,z' )-q^{-1}xm(-q^{-1}x^2,q^4,z')\\
&+\frac{z'J_2^3}{j(xz;q)j(z';q^4)}\Big [
\frac{j(-qx^2zz';q^2)j(z^2/z';q^{4})}{j(-qx^2z',z;q^2)}-xz \frac{j(-q^2x^2zz';q^2)j(q^2z^2/z';q^{4})}{j(-qx^2z',qz;q^2)}\Big ].\notag
\end{align}}
\end{corollary}

\begin{corollary} \label{corollary:msplitn2} For generic $x,z\in \mathbb{C}^*$ 
\begin{align}
m(x,q,z)&=m (-qx^2,q^4,-1 )-q^{-1}xm(-q^{-1}x^2,q^4,-1)\\
&\ \ \ -\frac{J_2^3}{j(xz;q)j(qx^2;q^2)\overline{J}_{0,4}}\Big [
\frac{j(qx^2z;q^2)j(-z^2;q^{4})}{j(z;q^2)}-xz \frac{j(q^2x^2z;q^2)j(-q^2z^2;q^{4})}{j(qz;q^2)}\Big ].\notag
\end{align}
\end{corollary}
 
We recall an identity \cite[Proposition $4.2$]{HM}, \cite[Theorem $2.2$]{H1}, which expresses the universal mock theta function in terms of Appell-Lerch sums:

\begin{proposition}  \label{proposition:HM-ID4.5} For generic $x,z\in\mathbb{C}^*$
\begin{align}
g(x,q)=-x^{-2}m(qx^{-3},q^3,x^3z)-x^{-1}m(q^2x^{-3},q^3,x^3z)+\frac{J_1^2j(xz;q)j(z;q^3)}{j(x;q)j(z;q)j(x^3z;q^3)}.
\end{align}
\end{proposition}
\noindent Taking the limit $z\rightarrow 1$ yields the following corollary:
\begin{corollary}\label{proposition:g-to-m} For generic $x\in\mathbb{C}^*$
\begin{align}
g(x,q)&=-x^{-1}m(q^2x^{-3},q^3,x^3)-x^{-2}m(qx^{-3},q^3,x^3)+\frac{J_3^3}{J_1j(x^3;q^3)}.\label{equation:newgid}
\end{align}
\end{corollary}
The following identity for $g(x,q)$ can be found in the lost notebook. 
\begin{proposition}{\cite[p. $32$]{RLN}, \cite[$(12.5.3)$]{ABI}} \label{theorem:gsplit} For generic $x$
\begin{equation}
g(x,q)=-x^{-1}+qx^{-3}g(-qx^{-2},q^4)-qg(-qx^2,q^4)+\frac{J_2J_{2,4}^2}{xj(x;q)j(-qx^2;q^2)}.
\end{equation}
\end{proposition}
\noindent Proposition \ref{theorem:gsplit} has a useful and easily shown corollary, the first half of which is also in the lost notebook \cite[p. $39$]{RLN}, \cite[$(12.4.4)$]{ABI}.
\begin{corollary} \label{corollary:rootsof1} For generic $x\in \mathbb{C}$
\begin{align}
g(x,q)+g(-x,q)&=-2qg(-qx^2,q^4)+\frac{2J_2\overline{J}_{1,4}^2}{j(-qx^2;q^4)j(x^2;q^2)},\label{equation:rootsof1n2k0}\\
g(x,q)-g(-x,q)&=-2x^{-1}+2qx^{-3}g(-qx^{-2},q^4)+\frac{2J_2\overline{J}_{1,4}^2}{xj(-q^3x^2;q^4)j(x^2;q^2)}.\label{equation:rootsof1n2k1}
\end{align}
\end{corollary}
\noindent We give a new proof of Proposition \ref{theorem:gsplit}.
\begin{proof}[Proof of Proposition \ref{theorem:gsplit}] We note the easily shown identity
\begin{equation}
g(x,q)=g(q/x,q)\label{equation:ginvt}.
\end{equation}
Applying Corollary (\ref{corollary:msplitn2zprime}) with $z'=q^6x^4$ to each Appell-Lerch sum of (\ref{equation:g-to-m}), we have
{\allowdisplaybreaks \begin{align}
g(x,q)&=-x^{-1}m(q^2x^{-3},q^3,x^2)-x^{-2}m(qx^{-3},q^3,x^2)\notag\\
&=-x^{-1}m(-q^7x^{-6},q^{12},q^6x^4)+q^{-1}x^{-4}m(-qx^{-6},q^{12},q^6x^4)\notag\\
& \ \ \ \ \ \ \ -\frac{q^6x^3J_6^3}{j(q^2x^{-1};q^3)j(-q^{13}x^{-2};q^6)j(q^6x^4;q^{12})}\cdot \frac{\overline{J}_{13,6}j(q^{-6};q^{12})}{j(x^2;q^6)}\notag\\
&\ \ \ \ -x^{-2}m(-q^5x^{-6},q^{12},q^6x^4)+q^{-2}x^{-5}m(-q^{-1}x^{-6},q^{12},q^6x^4)\notag\\
& \ \ \ \ \ \ \ -\frac{q^6x^2J_6^3}{j(qx^{-1};q^3)j(-q^{11}x^{-2};q^6)j(q^6x^4;q^{12})}\cdot \frac{\overline{J}_{11,6}j(q^{-6};q^{12})}{j(x^2;q^6)}.\label{equation:gexpA}
\end{align}}%
Using (\ref{equation:g-to-m}), we also have
\begin{align}
-qg(-qx^2,q^4)&=-qg(-q^3x^{-2},q^4)\notag\\
&=-q^{-2}x^2m(-q^{-1}x^6,q^{12},q^6x^{-4})+q^{-5}x^4m(-q^{-5}x^6,q^{12},q^6x^{-4})\notag\\
&=q^{-1}x^{-4}m(-qx^{-6},q^{12},q^6x^{4})-x^{-2}m(-q^{5}x^{-6},q^{12},q^6x^{4}),\label{equation:gexpB}
\end{align}
where the last line follows from (\ref{equation:m-fnq-flip}) and (\ref{equation:m-fnq-z}).  Similarly, we have
\begin{align}
qx^{-3}g&(-q^3x^2,q^4)=q^{-2}x^{-5}m(-q^{-1}x^6,q^{12},q^6x^{4})-q^{-5}x^{-7}m(-q^{-5}x^{-6},q^{12},q^6x^{4})\notag\\
&=q^{-2}x^{-5}m(-q^{-1}x^6,q^{12},q^6x^{4})+x^{-1}-x^{-1}m(-q^{7}x^{-6},q^{12},q^6x^{4}),\label{equation:gexpC}
\end{align}
where the last line follows from (\ref{equation:m-fnq-x-alt2}).  Combining (\ref{equation:gexpA}), (\ref{equation:gexpB}), and (\ref{equation:gexpC}) and then simplifying yields
\begin{align*}
g(x,q)&=-x^{-1}+qx^{-3}g(-qx^{-2},q^4)-qg(-qx^2,q^4)\\
& \ \ \ \ +\frac{x^{-1}J_6^3\overline{J}_{1,6}J_{6,12}}{j(q^6x^4;q^{12})j(x^2;q^6)}\cdot \Big [ \frac{1}{j(qx;q^3)j(-q^5x^2;q^6)}+\frac{x}{j(q^2x;q^3)j(-qx^2;q^6)}\Big ]\\
&=-x^{-1}+qx^{-3}g(-qx^{-2},q^4)-qg(-qx^2,q^4)\\
& \ \ \ \ +\frac{x^{-1}J_6^3\overline{J}_{1,6}J_{6,12}}{j(q^6x^4;q^{12})j(x^2;q^6)}\cdot \frac{j(-x;q^3)j(q^3x^2;q^6)}{j(qx;q^3)j(-q^5x^2;q^6)j(q^2x;q^3)j(-qx^2;q^6)}\cdot \frac{J_{2,12}^2J_{4,12}}{J_{12}^3},
\end{align*}
where the last equality follows from Proposition \ref{proposition:newgenid}.  The result then follows from elementary theta function properties.
\end{proof}


\subsection{Properties of Hecke-type double sums}
We recall from \cite{HM} some useful Hecke-type double sum identities:
\begin{proposition} \label{proposition:fabc-mod2} For $x,y\in\mathbb{C}^*$
{\allowdisplaybreaks {\begin{align}
f_{a,b,c}(x,y,q)&=f_{a,b,c}(-x^2q^a,-y^2q^c,q^4)-xf_{a,b,c}(-x^2q^{3a},-y^2q^{c+2b},q^4)\label{equation:fabc-mod2}\\
&\ \ \ \ -yf_{a,b,c}(-x^2q^{a+2b},-y^2q^{3c},q^4)+xyq^bf_{a,b,c}(-x^2q^{3a+2b},-y^2q^{3c+2b},q^4),\notag\\
f_{a,b,c}(x,y,q)&=-\frac{q^{a+b+c}}{xy}f_{a,b,c}(q^{2a+b}/x,q^{2c+b}/y,q),\label{equation:fabcflip}\\
f_{a,b,c}(x,y,q)& =-yf_{a,b,c}(q^bx,q^cy,q)+j(x;q^a),\label{equation:fspec-1}\\
f_{a,b,c}(x,y,q)& =-xf_{a,b,c}(q^ax,q^by,q)+j(y;q^c).\label{equation:fspec-2}
\end{align}}}
\end{proposition}

We state and prove a corollary to Theorem \ref{theorem:masterFnp}:
\begin{corollary}\label{corollary:cor-n3p2}For generic $x,y\in \mathbb{C}^*$
\begin{align*}
f_{3,5,3}(x,y,q)=g_{3,5,3}(x,y,q,-1,-1)+\theta_{3,2}(x,y,q),
\end{align*}
where
\begin{align}
g_{3,5,3}(x,y,q,-1,-1):=&\sum_{t=0}^{2}(-y)^tq^{3\binom{t}{2}}j(q^{5t}x;q^3)m\Big (-q^{27-16t}\cdot \frac{y^3}{x^5},q^{48},-1\Big )\label{equation:mdef-n3p2}\\
&+\sum_{t=0}^{2}(-x)^tq^{3\binom{t}{2}}j(q^{5t}y;q^3)m\Big (-q^{27-16t}\cdot \frac{x^3}{y^5},q^{48},-1\Big ),\notag
\end{align}
and
\begin{align}
\theta_{3,2}(x,y,q)&:=\frac{1}{\overline{J}_{0,48}}\cdot \frac{x^{1/2}y^{1/2}q^{-11/2}j(q^5xy;q^8)\overline{J}_{8,32}}{2j(q^5y^5/x^3;q^{16})j(q^5x^5/y^3;q^{16})}\label{equation:thetadef-n3p2}\\
& \ \ \ \cdot \Big [j(-q^{5/2}x^{5/2}/y^{3/2};q^{8})j(-q^{5/2}y^{5/2}/x^{3/2};q^{8})j(q^{3/2}y^{3/2}/x^{3/2};q^3)\notag \\ 
& \ \ \ \ \ \ -j(q^{5/2}x^{5/2}/y^{3/2};q^{8})j(q^{5/2}y^{5/2}/x^{3/2};q^{8})j(-q^{3/2}y^{3/2}/x^{3/2};q^3)\Big ].\notag
\end{align}
\end{corollary}

\begin{proof}[Proof of Corollary \ref{corollary:cor-n3p2}]  Identity (\ref{equation:mdef-n3p2}) easily follows from the definition (\ref{equation:mdef-2}).  For (\ref{equation:thetadef-n3p2}) we first substitute $n=3$ and $p=2$ into the definition of $\theta_{3,2}(x,y,q)$ to obtain
\begin{align*}
\theta_{3,2}&(x,y,q)=\Big [-\frac{y^2j(-x^3/y^3;q^{12})j(q^{10}x^2y^2;q^{32})}{q^4xj(q^5y^5/x^3,q^5x^5/y^3;q^{32})}
+\frac{y^3j(-q^6x^3/y^3;q^{12})j(q^{26}x^2y^2;q^{32})}{q^3xj(q^5y^5/x^3,q^{21}x^5/y^3;q^{32})}\\
& \ \ \ +\frac{q^3y^2j(-q^{-6}x^3/y^3;q^{12})j(q^{26}x^2y^2;q^{32})}{j(q^{21}y^5/x^3,q^5x^5/y^3;q^{32})}
-\frac{q^9y^3j(-x^3/y^3;q^{12})j(q^{42}x^2y^2;q^{32})}{j(q^{21}y^5/x^3,q^{21}x^5/y^3;q^{32})}\Big ]\cdot \frac{J_{32}^3}{\overline{J}_{0,48}}.
\end{align*}
We then combine the first and fourth summands as well as the second and third summands using (\ref{equation:H1Thm1.1}).  For example, we first use (\ref{equation:1.8}) to write
\begin{equation*}
j(q^{42}x^2y^2;q^{32})=-x^{-2}y^{-2}q^{-10}j(q^{10}x^2y^2;q^{32}).
\end{equation*}
This allows us to rewrite the sum of the first and fourth summands as
\begin{align*}
-&\frac{yj(-x^3/y^3;q^{12})j(q^{10}x^2y^2;q^{32})}{qx^2j(q^{5}y^5/x^3,q^5x^5/y^3,q^{21}y^5/x^3,q^{21}x^5/y^3;q^{32})}\\
&\ \ \ \ \ \ \ \ \ \ \cdot \Big [ j(q^5y^5/x^3,q^5x^5/y^3;q^{32})-xyq^{-3}j(q^{21}y^5/x^3,q^{21}x^5/y^3;q^{32})\Big ].
\end{align*}
We evaluate the bracketed expression by substituting $q\rightarrow q^{16}$, $x\rightarrow yq^{-3}$, $y\rightarrow -q^8y^4/x^4$,   in (\ref{equation:H1Thm1.1}), and we rewrite the denominator by using (\ref{equation:1.10}).  Using (\ref{equation:1.8}) to have
\begin{equation*}
j(xyq^{-3};q^{16})=-q^{-16}q^{-3}xyj(q^{13}xy;q^{16})=-q^{-3}xyj(q^{13}xy;q^{16}),
\end{equation*}
we see that the sum of the first and fourth summands is
\begin{align*}
-&\frac{y^2j(-x^3/y^3;q^{12})j(q^{10}x^2y^2;q^{32})j(q^{13}xy;q^{16})j(-q^8y^4/x^4;q^{16})J_{16}^2}{q^4xj(q^{5}y^5/x^3;q^{16})j(q^5x^5/y^3;q^{16})J_{32}^4}.
\end{align*}
The second and third summands can be combined in a similar way, and it follows that we can write
\begin{align}
\theta_{3,2}(x,y,q)&=\frac{1}{\overline{J}_{0,48}}\cdot \frac{x^3y^{-1}q^{-3}J_{16,32}}{j(q^5y^5/x^3;q^{16})j(q^5x^5/y^3;q^{16})}\label{equation:cor-n3p2A}\\
& \ \ \ \cdot \Big [j(-q^6x^3/y^3;q^{12})j(q^{26}x^2y^2;q^{32})j(-y^4/x^4;q^{16})j(q^5yx;q^{16})\notag \\ 
& \ \ \ \ \ \ -x^{-4}y^3q^{-1} j(-x^3/y^3;q^{12})j(q^{10}x^2y^2;q^{32})j(-q^8y^4/x^4;q^{16})j(q^{13}yx;q^{16})\Big ].\notag
\end{align}
Using (\ref{equation:1.12}) and (\ref{equation:1.10}) we obtain
\begin{align}
\theta_{3,2}(x,y,q)&=\frac{1}{\overline{J}_{0,48}}\cdot \frac{x^3y^{-1}q^{-3}}{j(q^5y^5/x^3;q^{16})j(q^5x^5/y^3;q^{16})}\notag\\
& \ \ \ \cdot \Big [j(-q^6x^3/y^3;q^{12})j(-q^{13}xy,q^{13}xy,-y^4/x^4,q^5yx;q^{16})\notag \\ 
& \ \ \ \ \ \ -x^{-4}y^3q^{-1} j(-x^3/y^3;q^{12})j(-q^{5}xy,q^5xy,-q^8y^4/x^4,q^{13}yx;q^{16})\Big ]\notag\\
&=\frac{1}{\overline{J}_{0,48}}\cdot \frac{x^3y^{-1}q^{-3}j(q^5xy;q^8)\overline{J}_{8,32}}
{j(q^5y^5/x^3;q^{16})j(q^5x^5/y^3;q^{16})}\label{equation:cor-n3p2B}\\
& \ \ \ \cdot \Big [j(-q^6x^3/y^3;q^{12})j(-q^{13}xy;q^{16})j(-y^4/x^4;q^{16})\notag \\ 
& \ \ \ \ \ \ -x^{-4}y^3q^{-1} j(-x^3/y^3;q^{12})j(-q^{5}xy;q^{16})j(-q^8y^4/x^4;q^{16})\Big ].\notag
\end{align}
From (\ref{equation:H1Thm1.2A}) and (\ref{equation:H1Thm1.2B}), we have
\begin{align}
j(q^{5/2}x^{5/2}/y^{3/2},q^{5/2}y^{5/2}/x^{3/2};q^8)&-j(-q^{5/2}x^{5/2}/y^{3/2},-q^{5/2}y^{5/2}/x^{3/2};q^8)\label{equation:1.2AA}\\
&\ \ \ =-2q^{5/2}x^{5/2}/y^{3/2}j(-y^4/x^4,-q^{13}xy;q^{16}),\notag\\
j(q^{5/2}x^{5/2}/y^{3/2},q^{5/2}y^{5/2}/x^{3/2};q^8)&+j(-q^{5/2}x^{5/2}/y^{3/2},-q^{5/2}y^{5/2}/x^{3/2};q^8)\label{equation:1.2BA}\\
&\ \ \ =2j(-q^8y^4/x^4,-q^{5}xy;q^{16}).\notag
\end{align}
Substituting (\ref{equation:1.2AA}) and (\ref{equation:1.2BA}) into (\ref{equation:cor-n3p2B}) and collecting terms yields
\begin{align}
\theta_{3,2}&(x,y,q)=\frac{1}{\overline{J}_{0,48}}\cdot \frac{x^{1/2}y^{1/2}q^{-11/2}j(q^5xy;q^8)\overline{J}_{8,32}}{2j(q^5y^5/x^3;q^{16})j(q^5x^5/y^3;q^{16})}\label{equation:cor-n3p2C}\\
& \cdot \Big [j(-q^{5/2}x^{5/2}/y^{3/2},-q^{5/2}y^{5/2}/x^{3/2};q^{8})\Big (j(-q^6x^3/y^3;q^{12}) -\frac{q^{3/2}y^{3/2}}{x^{3/2}}j(-x^3/y^3;q^{12}) \Big )\notag \\ 
&  \ \ \ \ \ -j(q^{5/2}x^{5/2}/y^{3/2},q^{5/2}y^{5/2}/x^{3/2};q^{8})\Big (j(-q^6x^3/y^3;q^{12}) +\frac{q^{3/2}y^{3/2}}{x^{3/2}}j(-x^3/y^3;q^{12}) \Big )\Big ].\notag
\end{align}
Using (\ref{equation:jsplit}) with $m=2$ gives the desired result.
\end{proof}

\subsection{Third order mock theta functions in terms of the Appell-Lerch sums.}\label{section:known}
We finish the preliminaries section by recalling the following identities from \cite{HM} which were shown using results of Watson \cite{W3} and Appell-Lerch sum properties. 
{\allowdisplaybreaks \begin{align}
\psi(q):&=\sum_{n\ge 1}\frac{q^{n^2}}{(q;q^2)_n}=qg(q,q^4)=-q^{-1}m(q,q^{12},q^2)-m(q^5,q^{12},q^2)\label{equation:3rdpsi}\\
&=-m(q,-q^3,-q)+\frac{qJ_{12}^3}{J_4 J_{3,12}}\notag\\
\nu(q):&=\sum_{n\ge 0} \frac{q^{n(n+1)}} {(-q;q^2)_{n+1}}=g(i\sqrt{q},q)
=q^{-1}m(q^2,q^{12},-q^3)+q^{-1}m(q^2,q^{12},-q^9)\label{equation:3rdnu}\\
&=2q^{-1}m(q^2,q^{12},-q^3)+\frac{J_1 J_{3,12}}{J_2}\notag\\
\phi(q):&=\sum_{n\ge 0}\frac{q^{n^2}}{(-q^2;q^2)_n}=(1-i)(1+ig(i,q))\label{equation:3rdphi}\\
&=m(q^5,q^{12},q^4)+m(q^5,q^{12},q^8)+q^{-1}m(q,q^{12},q^4)+q^{-1}m(q,q^{12},q^8)\notag\\
&=2m(q,-q^3,-1)+\frac{2qJ_{12}^3}{J_4 J_{3,12}} \notag 
\end{align}}




\section{Proof of Theorem \ref{theorem:newheckethm}}\label{section:proof-newheckethm}
We prove identity (\ref{equation:psi3(q)-hecke}).  Focusing on the right hand side, we have
 \begin{align*}
 \frac{1}{(q)_{\infty}}&\sum_{n= 0}^{\infty}(-1)^nq^{2n^2+n}(1+q^{2n+1})\sum_{j=-n}^nq^{-\binom{j+1}{2}}\\
&=\frac{1}{(q)_{\infty}}\Big ( \sum_{n= 0}^{\infty}(-1)^nq^{2n^2+n}\sum_{j=-n}^nq^{-\binom{j+1}{2}}+\sum_{n= 0}^{\infty}(-1)^nq^{2n^2+3n+1}\sum_{j=-n}^nq^{-\binom{j+1}{2}}\Big )\\
 &=\frac{1}{(q)_{\infty}}\Big ( \sum_{\substack{n+j\ge 0\\n-j\ge 0}}- \sum_{\substack{n+j< 0\\n-j< 0}} \Big )(-1)^nq^{2n^2+n-\binom{j+1}{2}},
\end{align*}
where in the last line we replaced $n$ with $-n-1.$  With a few more straightforward operations, we have
{\allowdisplaybreaks \begin{align*}
 \frac{1}{(q)_{\infty}}&\Big ( \sum_{\substack{n+j\ge 0\\n-j\ge 0}}- \sum_{\substack{n+j< 0\\n-j< 0}} \Big )(-1)^nq^{2n^2+n-\binom{j+1}{2}}\\
&=\frac{1}{(q)_{\infty}}\Big ( \sum_{\substack{r,s\ge 0\\r\equiv s \pmod 2}}- \sum_{\substack{r,s< 0\\r\equiv s \pmod 2}} \Big )(-1)^{\tfrac{r+s}{2}}q^{\tfrac{3}{8}r^2+\tfrac{5}{4}rs+\tfrac{3}{8}s^2+\tfrac{1}{4}r+\tfrac{3}{4}s}\\
&=\frac{1}{(q)_{\infty}}\Big (  f_{3,5,3}(q^2,q^3,q)-q^3f_{3,5,3}(q^6,q^7,q)\Big )\\
 &=\frac{1}{(q)_{\infty}}\Big (  2f_{3,5,3}(q^2,q^3,q)-(q)_{\infty}\Big ), &(\text{by }(\ref{equation:fspec-1}))
 \end{align*}}%
where the first equality follows from the substitutions $r=n+j$ and $s=n-j$, the second equality follows from considering the cases $r$, $s$ even and $r$, $s$ odd.
 So to prove (\ref{equation:psi3(q)-hecke}), it suffices to show
 \begin{equation}
 J_1\cdot (1+\psi(q))=f_{3,5,3}(q^2,q^3,q).
 \end{equation}
 We first compute $g_{3,5,3}(q^2,q^3,q,-1,-1)$.  Using Corollary \ref{corollary:cor-n3p2}, we have
 \begin{align*}
 g_{3,5,3}(q^2,q^3,q,-1,-1)&=j(q^2;q^3)m(-q^{26},q^{48},-1)+j(q^3;q^3)m(-q^{18},q^{48},-1)\\
 & \ \ -q^3 j(q^7;q^3)m(-q^{10},q^{48},-1)-q^2j(q^8;q^3)m(-q^{2},q^{48},-1)\\
  & \ \ +q^9 j(q^{12};q^3)m(-q^{-6},q^{48},-1)+q^7j(q^{13};q^3)m(-q^{-14},q^{48},-1)\\
  &=J_1m(-q^{26},q^{48},-1)-q^{-2}J_1m(-q^{10},q^{48},-1) \\
  &\ \ -q^{-5}J_1m(-q^2,q^{48},-1)
  +q^{-15}J_1m(-q^{-14},q^{48},-1),
\end{align*}
by (\ref{equation:1.8}).  Using (\ref{equation:m-fnq-x}) and (\ref{equation:m-fnq-flip}),
\begin{align}
  g_{3,5,3}(q^2,q^3,q,-1,-1)&=J_1+J_1q^{-22}m(-q^{-22},q^{48},-1)+q^{-12}J_1m(-q^{-10},q^{48},-1)\notag\\
  &\ \ +q^{-7}J_1m(-q^{-2},q^{48},-1)
  -q^{-1}J_1m(-q^{14},q^{48},-1)\notag\\
  &=J_1-J_1m(-q^{22},q^{48},-1)+q^{-12}J_1m(-q^{-10},q^{48},-1)\notag\\
  &\ \ +q^{-7}J_1m(-q^{-2},q^{48},-1)
  -q^{-1}J_1m(-q^{14},q^{48},-1).\label{equation:gpsi}
   \end{align}
 Using Corollary \ref{corollary:msplitn2}, we have
 \begin{align}
 m(q,q^{12},q^2)&=m(-q^{14},q^{48},-1)-q^{-11}m(-q^{-10},q^{48},-1)\notag\\
 &\ \ -\frac{J_{24}^3}{J_{3,12}J_{14,24}\overline{J}_{0,48}}\cdot \Big [ \frac{J_{16,24}\overline{J}_{4,48}}{J_{2,24}}-q^{3}\cdot \frac{J_{28,24}\overline{J}_{28,48}}{J_{14,24}}\Big ]\notag\\
 &=m(-q^{14},q^{48},-1)-q^{-11}m(-q^{-10},q^{48},-1)\notag\\
 &\ \ -\frac{J_{24}^3}{J_{3,12}J_{14,24}\overline{J}_{0,48}}\cdot \Big [ \frac{J_{16,24}\overline{J}_{4,48}}{J_{2,24}}+q^{-1}\cdot \frac{J_{4,24}\overline{J}_{28,48}}{J_{14,24}}\Big ],\label{equation:gsplitpsi1}
 \end{align}
 and
  {\allowdisplaybreaks \begin{align}
 m&(q^5,q^{12},q^2)=m(-q^{22},q^{48},-1)-q^{-7}m(-q^{-2},q^{48},-1)\notag\\
 &\ \ -\frac{J_{24}^3}{J_{7,12}J_{22,24}\overline{J}_{0,48}}\cdot \Big [ \frac{J_{24,24}\overline{J}_{4,48}}{J_{2,24}}-q^{7}\cdot \frac{J_{36,24}\overline{J}_{28,48}}{J_{14,24}}\Big ]\notag\\
 &=m(-q^{22},q^{48},-1)-q^{-7}m(-q^{-2},q^{48},-1)
  -q^{-5}\cdot \frac{J_{24}^3J_{12,24}\overline{J}_{28,48}}{J_{7,12}J_{22,24}\overline{J}_{0,48}J_{14,24}}.\label{equation:gsplitpsi2}
 \end{align}}%
 Substituting (\ref{equation:gsplitpsi1}) and (\ref{equation:gsplitpsi2}) into (\ref{equation:gpsi}) and using (\ref{equation:3rdpsi}), we have
 \begin{align*}
 g_{3,5,3}(q^2,q^3,q,-1,-1)&=J_1 (1+\psi(q)) -q^{-5}\cdot\frac{J_{24}^3J_1J_{12,24}\overline{J}_{28,48}}{J_{7,12}J_{22,24}\overline{J}_{0,48}J_{14,24}}\\
& \ \  -q^{-1}\cdot \frac{J_1J_{24}^3}{J_{3,12}J_{14,24}\overline{J}_{0,48}}\Big [ \frac{J_{16,24}\overline{J}_{4,48}}{J_{2,24}}+q^{-1}\cdot \frac{J_{4,24}\overline{J}_{28,48}}{J_{14,24}}\Big ]\\
&=J_1 (1+\psi(q)) -q^{-1}\cdot\frac{J_{24}^3J_{8,24}\overline{J}_{4,48}J_1}{J_{3,12}J_{10,24}\overline{J}_{0,48}J_{2,24}}\\
&\ \ -q^{-5}\cdot \frac{J_{24}^3\overline{J}_{28,48}J_1}{\overline{J}_{0,48}J_{10,24}}\Big [ \frac{J_{12,24}}{J_{7,12}J_{22,24}}+q^{3}\cdot \frac{J_{4,24}}{J_{3,12}J_{14,24}}\Big ],
\end{align*}
where in the last equality we grouped the first and third summands.  Using Proposition \ref{proposition:CHcorollary} with $q\rightarrow q^{24}$, $a= q^{13}$, $b= q^9$, $c= q^6$, $d= q$,  to evaluate the bracketed expression we obtain
\begin{align*}
 g_{3,5,3}&(q^2,q^3,q,-1,-1)=J_1 (1+\psi(q)) -q^{-1}\cdot\frac{J_{24}^3J_{8,24}\overline{J}_{4,48}J_1}{J_{3,12}J_{10,24}\overline{J}_{0,48}J_{2,24}}\\
&\ \ -q^{-5}\cdot \frac{J_{24}^3\overline{J}_{28,48}J_1}{\overline{J}_{0,48}J_{10,24}}\cdot \frac{J_{14,24}J_{8,24}J_{19,24}J_{17,24}}{J_{3,12}J_{7,12}J_{2,24}J_{10,24}}\cdot\frac{J_{12}}{J_{24}^2}\\
&=J_1 (1+\psi(q)) -q^{-1}\cdot\frac{J_1J_{24}J_8\overline{J}_{4,48}J_{12}}{J_{3,12}J_{2,12}\overline{J}_{0,48}}-q^{-5}\cdot \frac{J_1J_{24}J_8J_{12}\overline{J}_{28,48}}{\overline{J}_{0,48}J_{2,12}J_{3,12}}&(\text{by }(\ref{equation:1.10}))\\
&=J_1 (1+\psi(q)) -q^{-5}\cdot\frac{J_1J_{24}J_8J_{12}}{J_{3,12}J_{2,12}\overline{J}_{0,48}}\Big [ \overline{J}_{20,48}+q^4\overline{J}_{44,48}\Big ].
\end{align*}
Using (\ref{equation:jsplit}) with $m=2$ yields
\begin{align}
g_{3,5,3}(q^2,q^3,q,-1,-1)&=J_1 (1+\psi(q))-q^{-5}\cdot \frac{J_1J_{24}J_8J_{12}\overline{J}_{4,12}}{\overline{J}_{0,48}J_{2,12}J_{3,12}}.\label{equation:gpsi3-final}
 \end{align}
We compute $\theta_{3,2}(q^2,q^3,q)$.   Using Corollary \ref{corollary:cor-n3p2}, we have
 \begin{align}
 \theta_{3,2}(q^2,q^3,q)&=\frac{1}{2\overline{J}_{0,48}}\cdot \frac{q^{-3}j(q^{10};q^8)}{j(q^{14};q^{16})j(q^6;q^{16})}\cdot \frac{J_{16}^2}{J_8}\cdot \Big [ 0 - j(q^3;q^8)j(q^7;q^8)j(-q^3;q^3)\Big ]\notag\\
&=\frac{1}{2\overline{J}_{0,48}}\cdot \frac{q^{-5}J_{2,8}}{J_{2,16}J_{6,16}}\cdot \frac{J_{16}^2}{J_8}\cdot J_{3,8}J_{7,8}\overline{J}_{0,3}. \label{equation:thetapsi3-final}
 \end{align}
 Simplying with elementary theta function properties shows that (\ref{equation:thetapsi3-final}) is equal to the quotient of theta functions in (\ref{equation:gpsi3-final}), and the result follows.

We prove identity (\ref{equation:nu3(-q)-hecke}).  Focusing on the right hand side and replacing $q$ with $q^2$, we have
 \begin{align}
 \frac{1}{(q^2;q^2)_{\infty}}\sum_{n=0}^{\infty}(-1)^nq^{4n^2+4n}\sum_{j=-n}^nq^{-j^2-j}&=\frac{1}{2\cdot (q^2;q^2)_{\infty}}\Big ( \sum_{\substack{n+j\ge 0\\n-j\ge 0}}- \sum_{\substack{n+j< 0\\n-j< 0}} \Big )(-1)^nq^{4n^2+4n-j^2-j}\label{equation:alpha1}
 \end{align}
The quadratic part of the exponent in the Hecke-type sum factors.  Letting $j=2n-k,$ we can rewrite (\ref{equation:alpha1}) in terms of $k$ and $n$.  For fixed $k$, the sum over $n$ is a finite geometric series.  So (\ref{equation:alpha1}) is equivalent to
\begin{align}
\frac{1}{2\cdot (q^2;q^2)_{\infty}}\sum_{k=0}^{\infty}q^{-k^2+k}\cdot \frac{(-q^{4k+2})^{\lceil k/3 \rceil}-(-q^{4k+2})^{k+1}}{1+q^{4k+2}},\label{equation:alpha2}
\end{align}
where $\lceil \cdot \rceil$ is the ceiling function.  By standard series manipulations, we can rewrite (\ref{equation:alpha2}) as a sum of terms of the form $j*m$:
\begin{align*}
\frac{1}{2\cdot (q^2;q^2)_{\infty}}& \cdot \Big ( \overline{J}_{22,24}m(q^4,q^{24},-q^{22})-q^{-2}\overline{J}_{10,24}m(q^4,q^{24},-q^{10})
-\overline{J}_{6,24}m(q^{12},q^{24},-q^6)\\
&+\overline{J}_{6,24}m(q^{12},q^{24},-q^{18})-\overline{J}_{14,24}q^{-2}m(q^{4},q^{24},-q^{10})+\overline{J}_{2,24}m(q^{4},q^{24},-q^{22})\\
&+\overline{J}_{2,24}m(q^{4},q^{24},-q^{2})-\overline{J}_{10,24}q^{-2}m(q^{4},q^{24},-q^{14})+\overline{J}_{6,24}m(q^{12},q^{24},-q^{6})\\
&-\overline{J}_{14,24}q^{-2}m(q^{4},q^{24},-q^{14})+\overline{J}_{2,24}m(q^{4},q^{24},-q^{2})-\overline{J}_{6,24}m(q^{12},q^{24},-q^{-6})\Big ).
\end{align*}
This reduces to
\begin{align}
\frac{1}{J_2}\cdot \Big ( &\overline{J}_{22,24}m(q^4,q^{24},-q^{22})+\overline{J}_{2,24}m(q^{4},q^{24},-q^{2})-\overline{J}_{10,24}q^{-2}m(q^{4},q^{24},-q^{10})\label{equation:alpha3}\\
&-\overline{J}_{14,24}q^{-2}m(q^{4},q^{24},-q^{14})\Big ).\notag
\end{align}
Using identity (\ref{equation:m-change-z}), we have
{\allowdisplaybreaks \begin{align*}
m(q^4,q^{24},-q^{22})&=m(q^4,q^{24},q^6)+\frac{J_{24}^3\overline{J}_{16,24}\overline{J}_{8,24}}{\overline{J}_{2,24}^2J_{6,24}J_{10,24}},\\
m(q^4,q^{24},-q^{2})&=m(q^4,q^{24},q^6)+\frac{q^2J_{24}^3\overline{J}_{4,24}\overline{J}_{12,24}}{\overline{J}_{2,24}\overline{J}_{6,24}J_{6,24}J_{10,24}},\\
m(q^4,q^{24},-q^{10})&=m(q^4,q^{24},q^6)+\frac{q^6J_{24}^3\overline{J}_{4,24}\overline{J}_{20,24}}{\overline{J}_{10,24}^2J_{6,24}J_{10,24}},\\
m(q^4,q^{24},-q^{14})&=m(q^4,q^{24},q^6)+\frac{q^6J_{24}^3\overline{J}_{8,24}\overline{J}_{0,24}}{\overline{J}_{14,24}\overline{J}_{18,24}J_{6,24}J_{10,24}}.
\end{align*}}%
Using (\ref{equation:jsplit}), we have $J_2=\overline{J}_{10,24}-q^2\overline{J}_{22,24}.$  Thus (\ref{equation:alpha3}) is equivalent to
\begin{align}
-2q^{-2}m(q^4,q^{24},q^6)&+\frac{J_{24}^3}{J_2J_{6,24}J_{10,24}}\Big [ \frac{\overline{J}_{8,24}^2}{\overline{J}_{2,24}}-\frac{q^4\overline{J}_{4,24}^2}{\overline{J}_{10,24}}  \Big ]\notag \\
& +\frac{q^2J_{24}^3}{J_2J_{6,24}J_{10,24}}\Big [ \frac{\overline{J}_{4,24}\overline{J}_{12,24}}{\overline{J}_{6,24}}-\frac{q^2\overline{J}_{8,24}\overline{J}_{0,24}}{\overline{J}_{6,24}}  \Big ].\label{equation:alpha4}
\end{align}
Focusing on the second summand of (\ref{equation:alpha4}), we find that
\begin{align*}
\frac{J_{24}^3}{J_2J_{6,24}J_{10,24}}\Big [ \frac{\overline{J}_{8,24}^2}{\overline{J}_{2,24}}-\frac{q^4\overline{J}_{4,24}^2}{\overline{J}_{10,24}}  \Big ]&=\frac{J_{24}^3}{J_2J_{6,24}J_{10,24}}\cdot \frac{1}{\overline{J}_{2,24}\overline{J}_{14,24}}
\Big [ \overline{J}_{8,24}^2\overline{J}_{10,24}-q^4\overline{J}_{4,24}^2\overline{J}_{2,24}\Big ]\\
&=\frac{J_{24}^3}{J_2J_{6,24}J_{10,24}}\cdot \frac{1}{\overline{J}_{2,24}\overline{J}_{14,24}}
\cdot \frac{J_{10,24}^2J_{12,24}J_{4,24}}{\overline{J}_{6,24}}\\
&=\frac{J_{24}^3}{J_2J_{6,24}J_{10,24}}\cdot\frac{J_{10,24}^2J_{12}J_{2,12}}{\overline{J}_{6,24}J_{24}^2},
\end{align*}
where the second equality follows from Proposition \ref{proposition:CHcorollary} with $q\rightarrow q^{24}$, $a= q^{12}$, $b= q^2$, $c= -q^4$, $d= -q^8$, and the last equality follows from elementary theta function properties.  Focusing on the third summand of (\ref{equation:alpha4}), we obtain
\begin{align*}
\frac{q^2J_{24}^3}{J_2J_{6,24}J_{10,24}}\Big [ \frac{\overline{J}_{4,24}\overline{J}_{12,24}}{\overline{J}_{6,24}}-\frac{q^2\overline{J}_{8,24}\overline{J}_{0,24}}{\overline{J}_{6,24}}  \Big ]&=\frac{q^2J_{24}^3}{J_2J_{6,24}J_{10,24}\overline{J}_{6,24}}\cdot \Big [ \overline{J}_{4,24}\overline{J}_{12,24}-q^2\overline{J}_{8,24}\overline{J}_{0,24}\Big ]\\
&=\frac{q^2J_{24}^3J_{2,12}^2}{J_2J_{6,24}J_{10,24}\overline{J}_{6,24}},
\end{align*}
where the last line follows from (\ref{equation:H1Thm1.1}) with $q\rightarrow q^{12}$, $x= q^2$, $y= q^2$.  Assembling the pieces shows that (\ref{equation:alpha1}) is equivalent to
{\allowdisplaybreaks \begin{align*}
-2&q^{-2}m(q^4,q^{24},q^6)+\frac{J_{24}^3J_{2,12}}{J_2J_{6,24}J_{10,24}\overline{J}_{6,24}}\cdot \frac{J_{12}}{J_{24}^2}\cdot \Big [ q^2J_{2,12}\frac{J_{24}^2}{J_{12}}+J_{10,24}^2\Big ]\\
&=-2q^{-2}m(q^4,q^{24},q^6)+\frac{J_{24}^3J_{2,12}}{J_2J_{6,24}J_{10,24}\overline{J}_{6,24}}\cdot \frac{J_{12}}{J_{24}^2}\cdot \Big [ q^2J_{2,24}J_{14,24}+J_{10,24}^2\Big ]&(\text{by }(\ref{equation:1.10}))\\
&=-2q^{-2}m(q^4,q^{24},q^6)+\frac{J_{24}^3J_{2,12}}{J_2J_{6,24}J_{10,24}\overline{J}_{6,24}}\cdot \frac{J_{12}}{J_{24}^2}\cdot J_{10,24}\cdot \Big [ q^2J_{2,24}+J_{10,24}\Big ]\\
&=-2q^{-2}m(q^4,q^{24},q^6)+\frac{J_{24}^3J_{2,12}}{J_2J_{6,24}J_{10,24}\overline{J}_{6,24}}\cdot \frac{J_{12}}{J_{24}^2}\cdot J_{10,24}\cdot j(-q^2;-q^6).&(\text{by }(\ref{equation:jsplit}))
\end{align*}}
Elementary theta function properties shows that this is $\nu(-q^2).$

\section{Proof of Theorem \ref{theorem:newmockthetaid}}\label{section:proof-newmockthetaid}
We prove identity (\ref{equation:1.14mockthetaid}).  Focusing on the left hand side, we have
 \begin{align*}
 \sum_{n=0}^{\infty}\frac{q^{2n^2}}{(-q;q)_{2n}}&=\frac{1}{(q^2;q^2)_{\infty}}\sum_{n=0}^{\infty}q^{4n^2+n}(1-q^{6n+3})\sum_{j=-n}^n(-1)^jq^{-j^2}\\
 &=\frac{1}{ (q^2;q^2)_{\infty}}\Big ( \sum_{\substack{n+j\ge 0\\n-j\ge 0}}- \sum_{\substack{n+j< 0\\n-j< 0}} \Big )(-1)^jq^{4n^2+n- j^2}\\
&=\frac{1}{(q^2;q^2)_{\infty}}\Big (f_{3,5,3}(q^4,q^4,q^2)+q^5f_{3,5,3}(q^{12},q^{12},q^2)\Big )\\
&=\frac{1}{(q^2;q^2)_{\infty}}\Big (f_{3,5,3}(q^{5/4},-q^{5/4},-q^{1/2})\Big ). &(\text{by }(\ref{equation:fabc-mod2}))
 \end{align*}

\noindent We first compute $g_{3,5,3}(q^{5/4},-q^{5/4},-q^{1/2},-1,-1).$  Using Corollary \ref{corollary:cor-n3p2}, we have
\begin{align*}
g_{3,5,3}&(q^{5/4},-q^{5/4},-q^{1/2},-1,-1)=\Big [j(q^{5/4};-q^{3/2})+j(-q^{5/4};-q^{3/2})\Big ]m(-q^{11},q^{24},-1)\\
& \ \ \ \ +q^{5/4}\Big [j(-q^{15/4};-q^{3/2})-j(q^{15/4};-q^{3/2})\Big ]m(-q^{3},q^{24},-1)\\
& \ \ \ \ -q^{4}\Big [j(q^{25/4};-q^{3/2})+j(q^{25/4};-q^{3/2})\Big ]m(-q^{-5},q^{24},-1)\\
&=2J_2m(-q^{11},q^{24},-1)+2q^5J_{12,6}m(-q^{3},q^{24},-1)\\
& \ \ \ \ -2q^4J_{14,6}m(-q^{-5},q^{24},-1)\\
&=2J_2m(-q^{11},q^{24},-1)-2q^{-6}J_{2}m(-q^{-5},q^{24},-1),
\end{align*}
where the second equality follows from applying (\ref{equation:jsplit}) to each bracketed term, and the last equality follows from (\ref{equation:1.8}).  Using (\ref{equation:m-fnq-x-alt1}) and (\ref{equation:m-fnq-flip})
\begin{align*}
g_{3,5,3}&(q^{5/4},-q^{5/4},-q^{1/2},-1,-1)\\
&=2J_2(1+q^{-13}m(-q^{-13},q^{24},-1))+2q^{-1}J_{2}m(-q^{5},q^{24},-1)\\
&=2J_2 -2J_2m(-q^{13},q^{24},-1)+2q^{-1}J_{2}m(-q^{5},q^{24},-1).
\end{align*}
Using Proposition \ref{proposition:HM-ID4.5} with $q\rightarrow q^8$, $x= -q$, $z= q^{-3},$
\begin{align*}
-qg(-q,q^8)&=-m(-q^{13},q^{24},-1)+q^{-1}m(-q^{5},q^{24},-1)-\frac{qJ_{8}^2j(-q^{-2};q^8)j(q^{-3};q^{24})}{j(-q;q^8)j(q^{-3};q^8)j(-1;q^{24})}\\
&=-m(-q^{13},q^{24},-1)+q^{-1}m(-q^{5},q^{24},-1)-\frac{q^{-1}J_{8}^2\overline{J}_{2,8}J_{3,24}}{\overline{J}_{1,8}J_{3,8}\overline{J}_{0,24}}.
\end{align*}
So we have
\begin{equation}
g_{3,5,3}(q^{5/4},-q^{5/4},-q^{1/2},-1,-1)=2J_2-2J_2qg(-q,q^8)+\frac{2q^{-1}J_{8}^2J_2\overline{J}_{2,8}J_{3,24}}{\overline{J}_{1,8}J_{3,8}\overline{J}_{0,24}}.
\end{equation}
We compute $\theta_{3,2}(q^{5/4},-q^{5/4},-q^{1/2}).$  Using Corollary \ref{corollary:cor-n3p2}, we have
\begin{align*}
\theta_{3,2}(q^{5/4},-q^{5/4},-q^{1/2})&=\frac{q^{-3/2}}{2\overline{J}_{0,24}}\cdot \frac{j(q^5;q^4)\overline{J}_{4,16}}{j(q^5;q^8)^2}\\
& \ \ \ \  \cdot \Big [ j(q^{5/2};q^4)^2j(-q^{3/4};-q^{3/2})-j(-q^{5/2};q^{4})^2j(q^{3/4};-q^{3/2})\Big ]
\end{align*}
Using (\ref{equation:1.7}) and (\ref{equation:1.11}) yields
\begin{equation*}
j(-q^{3/4};-q^{3/2})=j(q^{3/4};-q^{3/2})=J_{3,6},
\end{equation*}
so we can write
\begin{align*}
\theta_{3,2}(q^{5/4},-q^{5/4},-q^{1/2})&=-\frac{q^{-5/2}}{2\overline{J}_{0,24}}\cdot\frac{J_{1,4}\overline{J}_{4,16}J_{3,6}}{j(q^5,q^8)^2}\cdot \Big [ j(q^{5/2};q^4)^2-j(-q^{5/2};q^{4})^2\Big ]\\
&=-\frac{q^{-5/2}}{2\overline{J}_{0,24}}\cdot\frac{J_{1,4}\overline{J}_{4,16}J_{3,6}}{j(q^5;q^8)^2}\\
& \ \ \ \ \cdot\Big ( j(q^{5/2};q^4)+j(-q^{5/2};q^{4})\Big )\cdot \Big ( j(q^{5/2};q^4)-j(-q^{5/2};q^{4})\Big )\\
&=\frac{q^{-5/2}}{2\overline{J}_{0,24}}\cdot\frac{J_{1,4}\overline{J}_{4,16}J_{3,6}}{j(q^5;q^8)^2}\cdot 2 j(-q^9;q^{16})\cdot 2q^{5/2}j(-q^{17};q^{16}),
\end{align*}
where the last line follows from applying (\ref{equation:jsplit}) with $m=2$ to each expression in parentheses.  Simplifying with elementary theta function properties gives
\begin{equation}
\theta_{3,2}(q^{5/4},-q^{5/4},-q^{1/2})=\frac{2q^{-1}J_{3,6}J_{16}J_8J_{2,16}}{\overline{J}_{0,24}J_{5,8}}.
\end{equation}
So proving (\ref{equation:1.14mockthetaid}) is equivalent to showing
\begin{equation}
\frac{2q^{-1}J_{8}^2\overline{J}_{2,8}J_{3,24}}{\overline{J}_{1,8}J_{3,8}\overline{J}_{0,24}}-\frac{2q^{-1}J_{3,6}J_{16}J_8J_{2,16}}{\overline{J}_{0,24}J_{5,8}J_2}=-\frac{J_{1,2}\overline{J}_{3,8}}{J_2}\label{equation:1.14final}.
\end{equation}
But elementary theta function properties shows that (\ref{equation:1.14final}) is equivalent to identity (\ref{equation:preidA}) of Proposition~\ref{lemma:newids}.

The proof of (\ref{equation:1.15mockthetaid}) is similar.  Here we find that
 \begin{align*}
  \sum_{n=0}^{\infty}\frac{q^{2n^2+2n}}{(-q;q)_{2n+1}}
 &=\frac{1}{ (q^2;q^2)_{\infty}}\Big ( \sum_{\substack{n+j\ge 0\\n-j\ge 0}}- \sum_{\substack{n+j< 0\\n-j< 0}} \Big )(-1)^jq^{4n^2+3n- j^2}\\
&=\frac{1}{J_2}\cdot f_{3,5,3}(q^{9/4},-q^{9/4},-q^{1/2}).
 \end{align*}
Using Corollary \ref{corollary:cor-n3p2} and arguing as above reduces proving (\ref{equation:1.15mockthetaid}) to identity (\ref{equation:preidB}) of Proposition~\ref{lemma:newids}.



\section{Proof of Theorem \ref{theorem:four-new-heckes}}\label{section:proofs-four-new-heckes}
Specializing Theorem \ref{theorem:main-acdivb}, we have that
\begin{align}
f_{4,4,1}(x,y,q)&=h_{4,4,1}(x,y,q,-1,-1)\notag\\
&\ \ \ \ \ -\sum_{d=0}^3\frac{q^{3\binom{d+1}{2}}j(q^{3+3d}y;q^4)j(-q^{9-3d}x/y;q^{12})J_{12}^3j(-q^{9+3d}/y^3;q^{12})}{\overline{J}_{0,3}\overline{J}_{0,12}j(-q^{6}x/y^4;q^{12})j(q^{3+3d}y/x;q^{12})},\label{equation:f441}
\end{align}
where
\begin{equation}
h_{4,4,1}(x,y,q,-1,-1)=j(x;q^4)m\big ( -q^3y/x,q^3,-1 \big )+j(y;q) m\big ( q^6 x/y, q^{12}, -1 \big ).\label{equation:h441}
\end{equation}
We prove identity (\ref{equation:hecke-phi}).   We first define
\begin{equation}
\sg(r,s):=\big (\sg(r)+\sg(s) \big)/2.
\end{equation}
We start with the right-hand side of (\ref{equation:hecke-phi}):
\begin{align*}
\sum_{n= 0}^{\infty}&(-1)^{n}q^{2n^2+n}(1+q^{2n+1})\sum_{j=-n}^n(-1)^jq^{-3j^2/2+j/2}\\
&=\sum_{\substack{n\ge 0 \\ -n\le j \le n}}(-1)^{n+j}q^{2n^2+n-3j^2/2+j/2}+q\sum_{\substack{n\ge 0 \\ -n\le j \le n}}(-1)^{n+j}q^{2n^2+3n-3j^2/2+j/2}\\
&=\sum_{n,j} {\sg}(j,n-j)(-1)^{n+j}q^{2n^2+n-3j^2/2+j/2}+q\sum_{n,j} {\sg}(j,n-j)(-1)^{n+j}q^{2n^2+3n-3j^2/2+j/2}\\
&=f_{4,4,1}(q^3,-q^2,q)+qf_{4,4,1}(q^5,-q^4,q),
\end{align*}
where the last line follows from the substitutions $u=j$, $v=n-j$.   

We first consider the Appell-Lerch sum expresssion.  Using (\ref{equation:h441}), we have
\begin{align}
h_{4,4,1}(q^3,-q^2,q)&+qh_{4,4,1}(q^5,-q^4,q)=j(q^3;q^4)m\big ( q^2,q^3,-1\big ) + j(-q^2;q) m\big ( q,q^{12},-1 \big )\notag\\
&+qj(q^5;q^4)m\big ( q^2,q^3,-1\big ) + qj(-q^4;q) m\big ( q^{-5},q^{12},-1 \big )\notag\\
=&2\overline{J}_{1,4}m\big ( q^5,q^{12},-1\big )+2q^{-1}\overline{J}_{1,4}m\big ( q,q^{12},-1\big ),
\end{align}
where the last line follows from (\ref{equation:1.8}) and (\ref{equation:m-fnq-flip}).  Thus
\begin{align}
f_{4,4,1}(q^3,-q^2,q)+qf_{4,4,1}&(q^5,-q^4,q)=2\overline{J}_{1,4}m\big ( q^5,q^{12},-1\big )+2q^{-1}\overline{J}_{1,4}m\big ( q,q^{12},-1\big )\notag\\
&-\sum_{d=0}^3\frac{q^{3\binom{d+1}{2}}j(-q^{5+3d};q^4)j(q^{10-3d};q^{12})J_{12}^3j(q^{3+3d};q^{12})}{\overline{J}_{0,3}\overline{J}_{0,12}j(-q;q^{12})j(-q^{2+3d};q^{12})}\notag\\
&-q\sum_{d=0}^3\frac{q^{3\binom{d+1}{2}}j(-q^{7+3d};q^4)j(q^{10-3d};q^{12})J_{12}^3j(q^{-3+3d};q^{12})}{\overline{J}_{0,3}\overline{J}_{0,12}j(-q^{-5};q^{12})j(-q^{2+3d};q^{12})}.\label{equation:hecke-phi-A}
\end{align}
Using (\ref{equation:m-change-z}) and grouping terms, we can rewrite (\ref{equation:hecke-phi-A}) as
\begin{align}
f_{4,4,1}&(q^3,-q^2,q)+qf_{4,4,1}(q^5,-q^4,q)\notag\\
&=\overline{J}_{1,4}m\big ( q^5,q^{12},q^4\big )+\overline{J}_{1,4}m\big ( q^5,q^{12},q^8\big )+q^{-1}\overline{J}_{1,4}m\big ( q,q^{12},q^4\big )+q^{-1}\overline{J}_{1,4}m\big ( q,q^{12},q^8\big )\notag\\
&\ \ \ \ \ +\frac{\overline{J}_{1,4}J_{12}^3\overline{J}_{4,12}}{\overline{J}_{0,12}\overline{J}_{5,12}J_{4,12}}\Big [ \frac{\overline{J}_{9,12}}{J_{9,12}}-\frac{\overline{J}_{1,12} } { J_{1,12} } \Big ]+q^{-1}\frac{\overline{J}_{1,4}J_{12}^3\overline{J}_{4,12}}{\overline{J}_{0,12}\overline{J}_{1,12}J_{4,12}}\Big [ \frac{\overline{J}_{5,12}}{J_{5,12}}+\frac{\overline{J}_{9,12} } { J_{9,12} } \Big ]\notag\\
& \ \ \ \ \  - \frac{1}{\overline{J}_{0,3}\overline{J}_{0,12}}\Big ( q^{-1}\frac{\overline{J}_{1,4}J_{12}^3J_{3,12}}{\overline{J}_{1,12}}\Big [ \frac{J_{2,12}}{\overline{J}_{2,12}}+\frac{J_{4,12}}{\overline{J}_{4,12}}\Big ]
+\frac{\overline{J}_{1,4}J_{12}^3J_{3,12}}{\overline{J}_{5,12}}\Big [ \frac{J_{4,12}}{\overline{J}_{4,12}}-\frac{J_{2,12}}{\overline{J}_{2,12}}\Big ]\notag\\
& \ \ \ \ \ \ \ \ \ \  +\frac{\overline{J}_{0,4}J_{12}^3J_{6,12}}{\overline{J}_{1,12}\overline{J}_{5,12}}\Big [ J_{1,12}+q^{-1}J_{7,12}\Big ]\Big ).\label{equation:hecke-phi-B}
\end{align}
With (\ref{equation:3rdphi}) in mind, to prove (\ref{equation:hecke-phi}) it remains to show that the sum of quotients of theta functions in (\ref{equation:hecke-phi-B}) is zero.  Using identities (\ref{equation:H1Thm1.2A}), (\ref{equation:H1Thm1.2B}), and (\ref{equation:jsplit}), the bracketed expressions can be evaluated and the terms can then be rearranged to show that the sum of quotients of theta functions is
\begin{align}
2q^{-1}&\frac{\overline{J}_{1,4}J_{12}^3\overline{J}_{4,12}J_{8,24}}{\overline{J}_{0,12}J_{3,12}J_{4,12}}\cdot \Big [ \frac{J_{14,24}}{\overline{J}_{1,12}J_{5,12}}-q^2\frac{J_{2,24}}{\overline{J}_{5,12}J_{1,12}}\Big ]\notag\\
& \ \  \ \ \ - 2q^{-1}\frac{\overline{J}_{1,4}J_{12}^3J_{3,12}J_{6,24}}{\overline{J}_{0,3}\overline{J}_{0,12}\overline{J}_{2,12}\overline{J}_{4,12}} \cdot \Big [ \frac{J_{14,24}}{\overline{J}_{1,12}}+q^3\frac{J_{2,24}}{\overline{J}_{5,12}}\Big ]-  \frac{\overline{J}_{0,4}J_{12}^3J_{6,12}}{\overline{J}_{0,3}\overline{J}_{0,12}\overline{J}_{1,12}\overline{J}_{5,12}}\cdot j(-q;-q^3)\notag\\
&=2q^{-1}\frac{\overline{J}_{1,4}J_{12}^3\overline{J}_{4,12}J_{8,24}}{\overline{J}_{0,12}J_{3,12}J_{4,12}}\cdot \frac{J_{24}}{J_{12}^2}\cdot\frac{J_{1,3}\overline{J}_{1,3}}{\overline{J}_{1,12}\overline{J}_{5,12}}\notag\\
& \ \  \ \ \ - 2q^{-1}\frac{\overline{J}_{1,4}J_{12}^3J_{3,12}J_{6,24}}{\overline{J}_{0,3}\overline{J}_{0,12}\overline{J}_{2,12}\overline{J}_{4,12}} \cdot \frac{1}{\overline{J}_{1,12}\overline{J}_{5,12}}\cdot \frac{\overline{J}_{3,12}J_{8}\overline{J}_{4,12}}{\overline{J}_{6,24}}\notag \\
&\ \ \ \ \ -  \frac{\overline{J}_{0,4}J_{12}^3J_{6,12}}{\overline{J}_{0,3}\overline{J}_{0,12}\overline{J}_{1,12}\overline{J}_{5,12}}\cdot j(-q;-q^3)
\label{equation:hecke-phi-C},
\end{align}
where the first bracketed expression was evaluated with (\ref{equation:1.12}) and (\ref{equation:jsplit}), and the second bracketed expression was evaluated with Proposition \ref{proposition:hecke-phi-id} with $q\rightarrow -q^{3}$ and $x\rightarrow -q^{-1}$.  Using a straightforward but lengthly argument with (\ref{equation:1.12}) and (\ref{equation:1.11}), showing that the right-hand side of (\ref{equation:hecke-phi-C}) is zero is equivalent to showing
\begin{equation}
2J_{12}^2J_{24}^2-J_{6,24}\overline{J}_{6,24}J_{12}J_{24}-J_{6,24}^2\overline{J}_{6,24}^2=0,
\end{equation}
which is straightforward.

We prove identity (\ref{equation:hecke-nu}).   We start with the right-hand side of (\ref{equation:hecke-nu}):
{\allowdisplaybreaks \begin{align*}
\sum_{n= 0}^{\infty}(-1)^nq^{2n^2+2n}\sum_{j=-n}^n(-1)^jq^{-3j^2/2+j/2}&=\sum_{\substack{n\ge 0 \\ -n\le j \le n}}(-1)^{n+j}q^{2n^2+2n-3j^2/2+j/2}\\
&=\sum_{n,j}\sg(j,n-j)(-1)^{n+j}q^{2n^2+2n-3j^2/2+j/2}\\
&=f_{4,4,1}(q^4,-q^3,q),
\end{align*}}%
where the last line follows from the substitutions $u=j$, $v=n-j$.  We first consider the Appell-Lerch sum expression.  Using (\ref{equation:h441}) we have
\begin{equation}
h_{4,4,1}(x,y,q)=j(q^4;q^4)m\big ( q^2,q^3,-1\big ) +j(-q^3;q)m\big( q^{-2},q^{12},-1\big )=2q^{-1}\overline{J}_{1,4}m\big (q^2,q^{12}-1\big ).
\end{equation}
Thus
\begin{align}
f_{4,4,1}(q^4,-q^3,q)&=2q^{-1}m\big (q^2,q^{12}-1\big )\notag\\
&\ \ \ \ \ - \sum_{d=0}^3\frac{q^{3\binom{d+1}{2}}j(-q^{6+3d};q^4)j(q^{10-3d};q^{12})J_{12}^3j(q^{3d};q^{12})}{\overline{J}_{0,3}\overline{J}_{0,12}j(-q^{-2};q^{12})j(-q^{2+3d};q^{12})}\label{equation:hecke-nu-A}
\end{align}
Using (\ref{equation:m-change-z}) and grouping terms, we can rewrite (\ref{equation:hecke-nu-A}) as
{\allowdisplaybreaks \begin{align}
f_{4,4,1}&(q^4,-q^3,q)=q^{-1}m\big ( q^2, q^{12},-q^3\big )+q^{-1}m\big( q^2,q^{12},-q^9\big )\notag\\
&\ \ \ \ \ +q^{-1}\frac{\overline{J}_{1,4}J_{12}^3J_{3,12}}{\overline{J}_{0,12}\overline{J}_{2,12}\overline{J}_{3,12}}\cdot \Big [ \frac{J_{5,12}}{\overline{J}_{5,12}}+\frac{J_{11,12}}{\overline{J}_{11,12}}\Big ]\notag\\
& \ \ \ \ \ -\frac{1}{\overline{J}_{0,3}\overline{J}_{0,12}} \cdot \Big ( q^{-1}\frac{\overline{J}_{1,4}J_{12}^3J_{3,12}}{\overline{J}_{2,12}} \cdot \Big [ \frac{J_{5,12}}{\overline{J}_{5,12}}+\frac{J_{11,12}}{\overline{J}_{11,12}}\Big ]+q^{-1}\frac{\overline{J}_{0,4}J_{4,12}J_{12}^3J_{6,12}}{\overline{J}_{2,12}\overline{J}_{8,12}}\Big )\label{equation:hecke-nu-B}
\end{align}}%
With (\ref{equation:3rdnu}) in mind, to prove (\ref{equation:hecke-nu}) it remains to show that the sum of quotients of theta functions in (\ref{equation:hecke-nu-B}) is zero.  Using identity (\ref{equation:H1Thm1.2B}), the bracketed expression can be evaluated and the terms can then be rearranged to show that the sum of quotients of theta functions in (\ref{equation:hecke-nu-B}) is
\begin{align}
&q^{-1}\frac{\overline{J}_{1,4}J_{12}^3J_{3,12}}{\overline{J}_{0,12}\overline{J}_{2,12}\overline{J}_{3,12}}\cdot \frac{2J_{16,24}J_{18,24}}{\overline{J}_{5,12}\overline{J}_{11,12}}\notag\\
& \ \ \ \ \ -\frac{1}{\overline{J}_{0,3}\overline{J}_{0,12}} \cdot \Big ( q^{-1}\frac{\overline{J}_{1,4}J_{12}^3J_{3,12}}{\overline{J}_{2,12}} \cdot \frac{2J_{16,24}J_{18,24}}{\overline{J}_{5,12}\overline{J}_{1,12}}+q^{-1}\frac{\overline{J}_{0,4}J_{4,12}J_{12}^3J_{6,12}}{\overline{J}_{2,12}\overline{J}_{8,12}}\Big ).
\label{equation:hecke-nu-C}
\end{align}
Using the identity $\overline{J}_{3,12}=2\overline{J}_{0,3},$ and elementary theta function properties, it is straightforward to verify that the right-hand side of (\ref{equation:hecke-nu-C}) is equal to zero.

We prove identity (\ref{equation:hecke-phibar0}).  Here we use the Bailey pair
\begin{equation}
B_n'(0,q):=1, \ \ \ \ \ A_n'(q^2,0,q):=\frac{q^{2n^2+3n}(1-q^{2n+2})}{(1-q)(1-q^2)}\sum_{j=-n-1}^n(-1)^jq^{-j(3j+5)/2}
\end{equation}
from Theorem $4$ and Lemma $6$ of \cite{A1}.  Bailey's lemma \cite[$(2.4)$]{A1} with $q\rightarrow q^2$, $\rho_1=-q^3$, $\rho_2=-q^2$, and $a=q^4$, then gives 
{\allowdisplaybreaks \begin{align}
J_{1,2}\overline{\phi_0}(q)&=\sum_{n=0}^{\infty}q^{4n^2+7n}(1-q^{2n+2})\sum_{j=-n-1}^n(-1)^jq^{-3j^2-5j}\notag\\
&=\Big ( \sum_{\substack{r,s\ge 0\\r\not\equiv s \pmod 2}}-\sum_{\substack{r,s< 0\\r\not\equiv s \pmod 2}}\Big )(-1)^{(r-s-1)/2}q^{(r+s+1)^2/4+(3r+2)s-1},\label{equation:BHL-id}
\end{align}}%
which is what one finds in \cite[$(2.7)$]{BHL}.  Replacing $(r,s)$ with $(2R+1,2S)$ and $(2R,2S+1)$, we obtain
\begin{align}
J_{1,2}\overline{\phi_0}(q)&=f_{1,7,1}(q^3,q^{13},q^2)-q^2f_{1,7,1}(q^9,q^7,q^2)\notag\\
&=-q^2f_{1,7,1}(q^{15},q^5,q^2)+q^4f_{1,7,1}(q^{9},q^{11},q^2)&{\text{(by (\ref{equation:fabcflip}))}}\notag\\
&=q^{-1}f_{1,7,1}(q,q^3,q^2)-q^{-1}J_{1,2}+q^4f_{1,7,1}(q^{9},q^{11},q^2)&{\text{(by (\ref{equation:fspec-1}))}}\notag\\
&=q^{-1}\sum_{n=0}^{\infty}q^{4n^2+n}(1-q^{6n+3})\sum_{j=-n}^n(-1)^jq^{-3j^2-j}-q^{-1}J_{1,2}.\notag
\end{align}

We prove identity (\ref{equation:hecke-phibar1}). Here we use the Bailey pair
\begin{equation}
B_n'(0,q):=1, \ \ \ \ \ A_n'(q,0,q):=\frac{q^{2n^2+n}(1-q^{2n+1})}{1-q}\sum_{j=-n}^n(-1)^jq^{-j(3j+1)/2}
\end{equation}
from Theorem $4$ and Lemma $7$ of \cite{A1}.  Bailey's lemma \cite[$(2.4)$]{A1} with $q\rightarrow q^2$, $\rho_1=-q$, $\rho_2=-q^2$, and $a=q^2$, then gives the result.

\section{Proof of Theorem \ref{theorem:BHLids}}\label{section:proof-BHLids}

We prove (\ref{equation:BHLphibar0}); the proof of (\ref{equation:BHLphibar1}) is similar.  We have
\begin{align*}
2q^2\overline{\phi}_0(q^2)&=\psi(q)+\psi(-q)&(\text{by (\ref{equation:BHLpsieven})})\\
&=qg(q,q^4)-qg(-q,q^4) &(\text{by (\ref{equation:3rdpsi})})\\
&=-2+2q^2g(-q^2,q^{16})+2\cdot \frac{J_8\overline{J}_{4,16}^2}{J_{2,8}\overline{J}_{14,16}}. &(\text{by (\ref{equation:rootsof1n2k1})})
\end{align*}
Elementary theta function properties gives the result.

\section*{Acknowledgements}

We would like to thank Dean Hickerson and the referee for helpful comments and suggestions.

\end{document}